\documentclass[11pt]{amsart}
\usepackage{amsmath,amssymb,amsthm, url, verbatim, wasysym}
\usepackage{graphicx}
\usepackage{color, float}
\usepackage[margin=1in]{geometry}

\theoremstyle{definition}
\newtheorem{defn}{Definition}[section]

\newtheorem{eg}[defn]{Example}
\theoremstyle{plain}

\newtheorem{lem}[defn]{Lemma} 

\usepackage{thmtools}
\usepackage{thm-restate}

\usepackage{hyperref}

\usepackage{cleveref}

\declaretheorem[name=Theorem,numberwithin=section]{thm}

\newcommand{\Mod}[1]{\ (\mathrm{mod}\ #1)}

\newcommand{\sk}{\mathcal{S}}
\renewcommand{\P}{\mathbf{P}}
\newcommand{\js}{\mathrm{js}}
\newcommand{\jx}{\mathrm{jx}}
\newcommand{\jd}{\mathrm{j}}
\newcommand{\jwn}{\vcenter{\hbox{\tiny{\def \svgwidth{.025\columnwidth}
\begingroup%
  \makeatletter%
  \providecommand\color[2][]{%
    \errmessage{(Inkscape) Color is used for the text in Inkscape, but the package 'color.sty' is not loaded}%
    \renewcommand\color[2][]{}%
  }%
  \providecommand\transparent[1]{%
    \errmessage{(Inkscape) Transparency is used (non-zero) for the text in Inkscape, but the package 'transparent.sty' is not loaded}%
    \renewcommand\transparent[1]{}%
  }%
  \providecommand\rotatebox[2]{#2}%
  \newcommand*\fsize{\dimexpr\f@size pt\relax}%
  \newcommand*\lineheight[1]{\fontsize{\fsize}{#1\fsize}\selectfont}%
  \ifx\svgwidth\undefined%
    \setlength{\unitlength}{59.54299851bp}%
    \ifx\svgscale\undefined%
      \relax%
    \else%
      \setlength{\unitlength}{\unitlength * \real{\svgscale}}%
    \fi%
  \else%
    \setlength{\unitlength}{\svgwidth}%
  \fi%
  \global\let\svgwidth\undefined%
  \global\let\svgscale\undefined%
  \makeatother%
  \begin{picture}(1,1.51151269)%
    \lineheight{1}%
    \setlength\tabcolsep{0pt}%
    \put(0,0){\includegraphics[width=\unitlength,page=1]{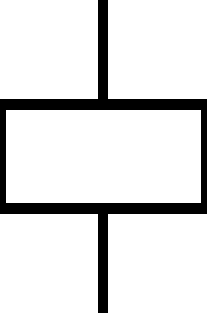}}%
  \end{picture}%
\endgroup%
}}}}

\usepackage{tikz}
\newcommand\TikCircle[1][2.5]{\tikz[baseline=-#1]{\draw[thick](0,0)circle[radius=#1mm];}}
\newcommand\ocrossing[1][2.5]{\tikz[baseline=-#1]{\draw[thick](0,0) --(.3, .3); \draw[thick](0, .3) --(.09, .2); \draw[thick](.2, .1) --(.3, 0)}}

\newcommand\ares[1][2.5]{\tikz[baseline=-#1]{\draw[thick] (0, 0.3) to [out=315, in=45] (0, 0); \draw[thick] (0.3, 0) to [out=135, in=225] (.3, .3)}}

\newcommand\bres[1][2.5]{\tikz[baseline=-#1]{\draw[thick] (0, 0.3) to [out=315, in=225] (.3, .3); \draw[thick] (0.3, 0) to [out=135, in=45] (0, 0)}}

\begin{document}
\title{Cancellations in the degree of the colored Jones polynomial}
\author[C. Lee]{Christine Ruey Shan Lee}
\address[]{Department of Mathematics and Statistics, University of South Alabama, Mobile AL 36608}
\email[]{christine.rs.lee@gmail.com}
\author[R. van der Veen]{Roland van der Veen}
\address[]{University of Groningen, Bernoulli Institute, P.O. Box 407, 9700 AK Groningen, The Netherlands}
\email[]{r.i.van.der.veen@rug.nl}

\thanks{C. Lee is supported in part by a National Science Foundation grant DMS 1907010. }

\begin{abstract}
We give an alternate expansion of the colored Jones polynomial of pretzel links which recovers the degree formula in \cite{GLV17}. As an application, we determine the degrees of the colored Jones polynomials of a new family of 3-tangle pretzel knots. 
\end{abstract}
\maketitle

\section{introduction} 
The discovery of the Jones polynomial and related quantum knot and 3-manifold invariants has revolutionized knot theory. The colored Jones polynomial, an important example that is an invariant for a link in $S^3$ from the representation theory of the quantum group $U_q(\mathfrak{sl}_2)$, has been studied a great deal. 

For alternating and more generally adequate knots, it has been shown that the coefficients of their colored Jones polynomial stabilize \cite{DL07, Arm13, GL15} and give information on the hyperbolic volume and geometry of the knot complement \cite{DL07, FKP13}, and that their degrees relate to the boundary slopes of essential surfaces \cite{Ga:slope, KT15}. 

For non-adequate knots,  it is natural to ask the extent to which similar results hold. Results extending relationships observed for adequate knots exist \cite{GV, LV:3pretzel, MT, BMT18, Lee17,  LLY, Howie}.  However, it is  difficult to study the colored Jones polynomial in complete generality, since the state sum which may be used to define the polynomial often has cancellations that are difficult to control. The purpose of this paper is to give an expansion of the colored Jones polynomial whose cancellation is more explicitly managed. The main ingredient is Khovanov's coefficient formulas for the expansion of the Jones-Wenzl projector in the canonical basis of the Temperley-Lieb algebra \cite{Kho97}. 

We will consider links in $S^3$ with diagrams from the projection to $S^2$. Let $P(w_0, w_1, \ldots, w_m)$ denote the pretzel link with standard diagram $L$ consisting of $m+1$ vertical twist regions, each of $|w_i|$ crossings joined side by side, with the twist region of $w_{m}$ joined to that of  $w_{0}$ as shown on the left in Figure \ref{f.pretzel}. If $w_i > 0$ then the twist region is made up of positive crossings, and if $w_i< 0$, the twist region is made up of negative crossings.
\begin{figure}[H]
\def \svgwidth{\columnwidth}
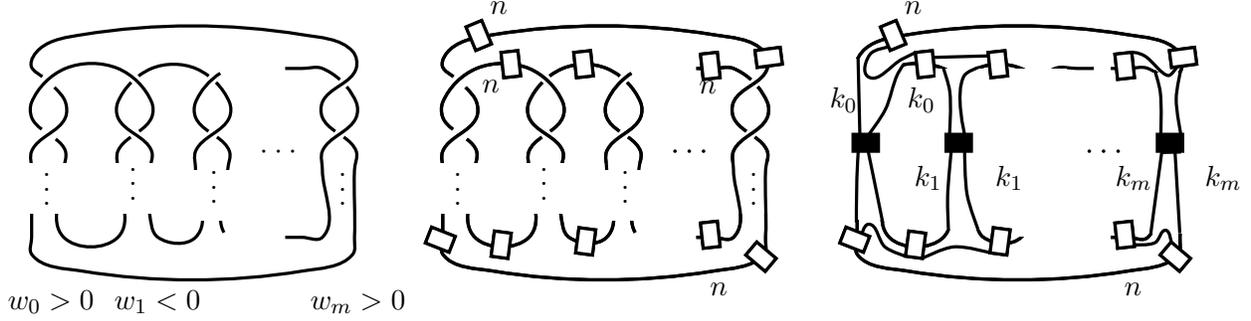 
\caption{\label{f.pretzel} Left: A diagram $L$ for the pretzel link $P(w_0, w_1, \ldots, w_m)$ with positive and negative twist regions. Middle: The $n$-cabled skein element $L^n$ decorated by Jones-Wenzl projectors. Right: The skein element $T$ resulting from applying the fusion and untwisting formulas to $L^n$. The center Jones-Wenzl projector to each twist region is shown as a black box.}
\end{figure} 
Consider the diagram $L$ as an element in the skein module $K(\mathbb{R}^2)$ of the plane $\mathbb{R}^2$. See Definition \ref{d.skein}. For a fixed positive integer $n\geq 1$, cable a Jones-Wenzl projector as defined in Definition \ref{d.jw} and shown as a box  $\vcenter{\hbox{\tiny{\def \svgwidth{.025\columnwidth}
\begingroup%
  \makeatletter%
  \providecommand\color[2][]{%
    \errmessage{(Inkscape) Color is used for the text in Inkscape, but the package 'color.sty' is not loaded}%
    \renewcommand\color[2][]{}%
  }%
  \providecommand\transparent[1]{%
    \errmessage{(Inkscape) Transparency is used (non-zero) for the text in Inkscape, but the package 'transparent.sty' is not loaded}%
    \renewcommand\transparent[1]{}%
  }%
  \providecommand\rotatebox[2]{#2}%
  \newcommand*\fsize{\dimexpr\f@size pt\relax}%
  \newcommand*\lineheight[1]{\fontsize{\fsize}{#1\fsize}\selectfont}%
  \ifx\svgwidth\undefined%
    \setlength{\unitlength}{59.54299851bp}%
    \ifx\svgscale\undefined%
      \relax%
    \else%
      \setlength{\unitlength}{\unitlength * \real{\svgscale}}%
    \fi%
  \else%
    \setlength{\unitlength}{\svgwidth}%
  \fi%
  \global\let\svgwidth\undefined%
  \global\let\svgscale\undefined%
  \makeatother%
  \begin{picture}(1,1.51151269)%
    \lineheight{1}%
    \setlength\tabcolsep{0pt}%
    \put(0,0){\includegraphics[width=\unitlength,page=1]{jwn.pdf}}%
  \end{picture}%
\endgroup%
}}}_n$ on $n$ strands in each component of $L$ using the blackboard framing. Denote the resulting skein element by $L^n$. Using the idempotent property of the projector, double the projectors so that each twist region is framed by four projectors, then apply the fusion and untwisting formulas \eqref{e.funtwist} to each of the twist regions of $P(w_0, w_1, \ldots, w_m)$. Let $k = (k_0, k_1, \ldots, k_m)$ be the set of fusion parameters. We get a skein element $T_k \in K(\mathbb{R}^2)$ that is the union of skein elements $T_i$ in the Temperley-Lieb algebra  $TL_n$ for  $0\leq i \leq m$ decorated by Jones-Wenzl projectors joined side by side. For each $T_i$ there is a center Jones-Wenzl projector from fusing and untwisting the twist region, and four framing projectors shared by adjacent skein elements $T_{i-1}$, $T_{i+1}$.  See Figure \ref{f.pretzel} for an illustration. 

The Jones-Wenzl projector in $TL_n$ has an expansion
\[ \vcenter{\hbox{\tiny{\def \svgwidth{.025\columnwidth}}}}_n = \sum_{d\in B_n^{TL}} \mathbf{P}(d) d \] 
 in the canonical basis $B_n^{TL}$, where $\mathbf{P}(d)$ is the coefficient multiplying the skein element $d$ of the basis in the expansion, see Definition \ref{d.dex}. Let $\sigma$ denote the choice of the skein element $\sigma = (d_1, \ldots, d_m, \sigma_t^1, \sigma_b^1, \ldots, \sigma_t^{m-1}, \sigma_b^{m-1})$ in the expansion of the Jones-Wenzl projectors decorating $T_k$. More precisely, for $1\leq i \leq m$, $d_i \in B_n^{TL}$ is the choice of a skein element in the expansion of the center projector for $T_i$.  For $1\leq i \leq m-1$, $\sigma_t^i \in B_n^{TL}$ or $\sigma_b^i \in B_{n_i}^{TL}$ is the choice, in order, of a skein element of the top or bottom projector shared by $T_{i}$, $T_{i+1}$, where after choosing $\sigma_t^i$, we remove any circles attached to a projector via \eqref{e.rcircle}, then choose $\sigma_b^i$ for the bottom projector in $TL_{n_i}$. If the number of circles removed is $c_i$, then  $n_i = n-c_i$.  For $n\geq 1$, we may write the $n+1$ colored Jones polynomial of the pretzel link as follows. See Definition \ref{d.cjp}. 
\begin{equation} \label{e.ssum}
\begin{split}
J_{P, n+1} &= ((-1)^n q^{\frac{1}{2}})^{w(L)(n^2+2n)} \langle L^n \rangle \text{, where} \\ 
\langle L^n \rangle  &= \sum_{0 \leq k_i \leq n, \ \sigma}  \underbrace{\left( \prod_{i=0}^m \frac{[2k_i+1]}{\theta(n, n, 2k_i)} U(w_i, k_i) \right) \left(\prod_{i=1}^m\mathbf{P}(d_i) \right) \left(  \prod_{i=1}^{m-1}  \mathbf{P}(\sigma_t^i)\mathbf{P} (\sigma_b^i)R(c_i)  \right)}_{G_{k, \sigma}}  \langle  T^n_{k, \sigma} \rangle.  \end{split}
\end{equation} 
Here $T^n_{k, \sigma}$ is the skein element that comes from applying $\sigma$ to $T_k$, $\langle \cdot \rangle$ is the Kauffman bracket as in Definition \ref{d.kbracket}, $w(L)$ is the writhe of the diagram $L$, $\frac{[2k_i+1]}{\theta(n, n, 2k_i)}U(w_i, k_i)$ is a product of rational functions in $q$ from the fusion and untwisting formulas, and $R(c_i)$ is the rational function resulting from removing $c_i$ circles from a Jones-Wenzl projector in $TL_{n_i}$ using \eqref{e.rcircle}. 

Denote the rational function in $q$ in \eqref{e.ssum} multiplying $\langle T^n_{k, \sigma}  \rangle$ by $G_{k, \sigma}$.  Our main result is the explicit formula of $G_{k, \sigma}$ in terms of quantum factorials for the leading terms of the state sum defining the colored Jones polynomial of a pretzel link. 
\begin{thm} \label{t.mainintro}
Let $k = (k_0, k_1, \ldots, k_m) \in \mathbb{Z}_{\geq 0}^{m+1}$ with $k_i \leq n$ and $\sigma =  (d_1, \ldots, d_m, \sigma_t^1, \ldots, \sigma_b^{m-1})$ as described above. 
The  $n+1$ colored Jones polynomial of the pretzel link $P=P(w_0, w_1, \ldots, w_m)$ with standard diagram $L$ has the form 
\[  J_{P, n+1} =((-1)^n q^{\frac{1}{2}})^{w(L)(n^2+2n)}\sum_{0\leq k_i\leq n, \sigma} G_{k, \sigma} \langle T^n_{k, \sigma} \rangle  = ((-1)^n q^{\frac{1}{2}})^{w(L)(n^2+2n)} \sum_{k_0 = \sum_{i=1}^m k_i } \mathcal{G}_{k},\]
where 
\begin{align*} 
 \mathcal{G}_{k} &=\left(\prod_{i=0}^m \frac{[2k_i+1]}{\theta(n, n, 2k_i)} U(w_i, k_i)\right) \left( \prod_{i=1}^{m-1} \left( \frac{[n-\sum_{j=1}^i k_j]![n-k_{i+1}]!}{[n-\sum_{j=1}^{i+1} k_j]![n]!} \right)^2 \right) \langle \mathcal{T}_{k_0} \rangle + l.o.t.\end{align*} 
Here l.o.t. denotes lower order terms, and $\langle \mathcal{T}_{k_0} \rangle$ is the Kauffman bracket of the following skein element: 
\begin{figure}[H]
\def \svgwidth{.3\columnwidth}
\begingroup%
  \makeatletter%
  \providecommand\color[2][]{%
    \errmessage{(Inkscape) Color is used for the text in Inkscape, but the package 'color.sty' is not loaded}%
    \renewcommand\color[2][]{}%
  }%
  \providecommand\transparent[1]{%
    \errmessage{(Inkscape) Transparency is used (non-zero) for the text in Inkscape, but the package 'transparent.sty' is not loaded}%
    \renewcommand\transparent[1]{}%
  }%
  \providecommand\rotatebox[2]{#2}%
  \newcommand*\fsize{\dimexpr\f@size pt\relax}%
  \newcommand*\lineheight[1]{\fontsize{\fsize}{#1\fsize}\selectfont}%
  \ifx\svgwidth\undefined%
    \setlength{\unitlength}{509.85356161bp}%
    \ifx\svgscale\undefined%
      \relax%
    \else%
      \setlength{\unitlength}{\unitlength * \real{\svgscale}}%
    \fi%
  \else%
    \setlength{\unitlength}{\svgwidth}%
  \fi%
  \global\let\svgwidth\undefined%
  \global\let\svgscale\undefined%
  \makeatother%
  \begin{picture}(1,0.80622931)%
    \lineheight{1}%
    \setlength\tabcolsep{0pt}%
    \put(0,0){\includegraphics[width=\unitlength,page=1]{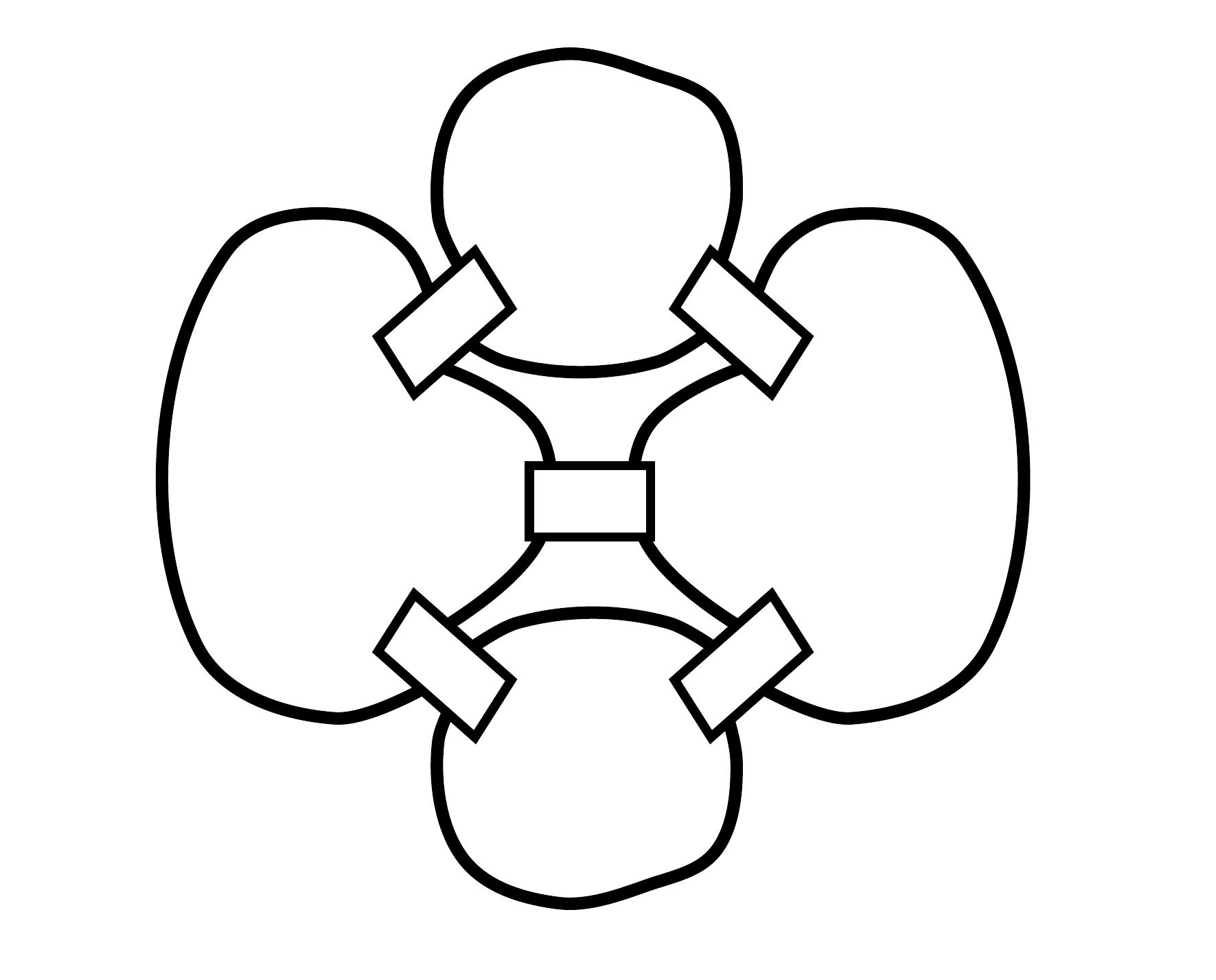}}%
    \put(0.56414705,0.43182641){\makebox(0,0)[lt]{\lineheight{1.25}\smash{\begin{tabular}[t]{l}$k_0$\end{tabular}}}}%
    \put(0.3196507,0.43182641){\makebox(0,0)[lt]{\lineheight{1.25}\smash{\begin{tabular}[t]{l}$k_0$\end{tabular}}}}%
    \put(0.56505062,0.33179779){\makebox(0,0)[lt]{\lineheight{1.25}\smash{\begin{tabular}[t]{l}$k_0$\end{tabular}}}}%
    \put(0.3234963,0.33473982){\makebox(0,0)[lt]{\lineheight{1.25}\smash{\begin{tabular}[t]{l}$k_0$\end{tabular}}}}%
    \put(0.87963755,0.42300039){\makebox(0,0)[lt]{\lineheight{1.25}\smash{\begin{tabular}[t]{l}$k_0$\end{tabular}}}}%
    \put(-0.00296884,0.45242058){\makebox(0,0)[lt]{\lineheight{1.25}\smash{\begin{tabular}[t]{l}$k_0$\end{tabular}}}}%
    \put(0.39420396,0.77604292){\makebox(0,0)[lt]{\lineheight{1.25}\smash{\begin{tabular}[t]{l}$n-k_0$\end{tabular}}}}%
    \put(0.41731984,0.00691452){\makebox(0,0)[lt]{\lineheight{1.25}\smash{\begin{tabular}[t]{l}$n-k_0$\end{tabular}}}}%
  \end{picture}%
\endgroup%
 
\caption{\label{f.tko} The skein element $\mathcal{T}_{k_0}$}
\end{figure}  
\end{thm} 
See Equation \eqref{e.binom}  for the definition of the quantum factorial.  As a direct corollary, we recover the degree formula\footnote{There is a slight difference in notation from the cited theorem where $\mathcal{G}_{c, k}$ there is our $\mathcal{G}_{k}$.} from \cite{GLV17}.
\begin{restatable}{thm}{gssr} (\cite[Theorem 3.2]{GLV17}) \label{t.genstatesum} Let $k = (k_0, k_1, \ldots, k_m) \in \mathbb{Z}_{\geq 0}^{m+1}$ with $ k_i \leq n$ and assume $|w_i| > 1$. The $n+1$ colored Jones polynomial of a pretzel link $P = P(w_0, w_1, \ldots, w_m)$ with diagram $L$ has the form 
\[ J_{P, n+1} =  ((-1)^n q^{\frac{1}{2}})^{w(L)(n^2+2n)}\sum_{k_0 = \sum_{i=1}^m k_i } \mathcal{G}_{k}, \]
where $ \deg \mathcal{G}_{k} = (-1)^{w_0(n-k_0)+n+k_0+\sum_{i=1}^m (n-k_i) (w_i-1)} q^{\delta(n, k)} + l.o.t.$\footnote{denotes lower order terms.},  and  $\delta(n, k) = $
\[ -\left( (w_0+1)k^2_0 + \sum_{i=1}^m (w_i-1)k_i^2 + \sum_{i=1}^m (-2+w_0 + w_i)k_i - \frac{n(n+2)}{2} \sum_{i=0}^m w_i + (m-1)n \right).   \]  
\end{restatable} 
The advantage of our approach is that each term of the state sum is explicit as described by \eqref{e.ssum} even if they are lower order. For 3-tangle pretzel knots $P = P(w_0, w_1, w_2)$ we show the following.

\begin{restatable}{thm}{pretzel} \label{t.3pretzel} 
Let $w = (w_0, w_1, w_2)\in \mathbb{Z}^3$ be such that $w_0 < -1 < 0 < 1 < w_1, w_2$. Define \[s(w) =  1+w_0+\frac{1}{\sum_{i=1}^2 (w_i-1)^{-1}} \text{ and } s_1(w) =  \frac{\sum_{i=1}^2(w_i + w_0-2)(w_i-1)^{-1}}{\sum_{i=1}^2 (w_i-1)^{-1}}. \] 
Suppose $w_1$ is even and $-w_0 > \min\{w_1-1, w_2-1\}$. Let $P$ denote the pretzel knot $P(w_0, w_1, w_2)$, and let $\jd_P(n)$ be the largest power of $q$ in $J_{P, n}$.   If  $s(w) < 0$, we may write
\[ \jd_P(n) = \js_Pn^2 + \jx_P(n)n+c_P(n),\] where $\jx_P$, $c_P$ are periodic functions in $n$. 
In particular we have: 
\begin{itemize} 
 \item[(a)] For $n =\frac{-2+w_1+w_2}{\gcd(w_1-1, w_2-1)}j$,  $j\geq1$:
\begin{equation} 
 \js_P = -s(w)+w_0 + w_2, \ \jx_P(n) = -s_1(w) + 2s(w) - (m-1) - 2\frac{\min{\{w_1-1, w_2-1\}}}{-2+w_1+w_2} .   \label{e.cancel} \end{equation} 
\item[(b)] For $n \not=\frac{-2+w_1+w_2}{\gcd(w_1-1, w_2-1)}j$:
\begin{equation} \label{e.regular}
  \js_P = -s(w)+w_0 + w_2, \ \jx_P(n) = -s_1(w) + 2s(w) - (m-1). 
 \end{equation} 
 \end{itemize} 
\end{restatable} 
 \begin{eg} The pretzel knot $P(-5, 4, 3)$. Note $s(w) = -\frac{14}{5} < 0$ and $s_1(w) = -\frac{18}{5} $.  Since \[\frac{\min\{w_1-1, w_2-1\}}{\gcd(w_1-1, w_2-1)} =\frac{\min\{4-1, 3-1\}}{\gcd(4-1, 3-1)} = 2, \] we have by Theorem \ref{t.3pretzel}, 
 for $n =  0 \Mod 5$, $\js_P =\frac{4}{5}$, and $\jx_P(n) = -3-\frac{4}{5}$. Otherwise for $n \not=  0 \Mod 5$, $\js_P =\frac{4}{5}$, and $\jx_P(n) = -3$. 
 \end{eg} 
The explicitly calculated \emph{cancellation} of the terms in the expansion \eqref{e.ssum} of the $n$ colored Jones polynomial is the decrease by $2\frac{\min{\{w_1-1, w_2-1\}}}{-2+w_1+w_2}$ in $\jx_P(n)$ in \eqref{e.cancel}  for certain congruence of $n$. This is of interest as there is yet to be an explanation for the deviation of $\jx_P(n)$ from the Euler characteristic of an essential surface in the complement of $P$ from the viewpoint of the strong slope conjecture, see  \cite{Ga:slope, KT15}. 
We expect that the expansion would also apply to find explicit stable coefficients studied in \cite{Arm13, GL15}, and we will address this question in the future. 

This paper is organized as follows: In Section \ref{s.exjw} we summarize the background necessary to understand \eqref{e.ssum} by expanding the Jones-Wenzl projectors in the canonical basis. In Section \ref{s.org}, we prove lemmas useful for comparing degrees of $G_{k, \sigma}$. Then we prove Theorem \ref{t.mainintro} in Section \ref{s.degp}, show how Theorem \ref{t.genstatesum} follows, and conclude with the proof of Theorem \ref{t.3pretzel} in Section  \ref{ss.3pretzel}. 

\section{Preliminaries} \label{s.exjw}
 We follow \cite{Kho97}. 
We define for non-negative integers $n, k$, and indeterminate $q$,
\begin{equation}  \label{e.binom}
[n] = \frac{q^n - q^{-n}}{q-q^{-1}}, \qquad [n]! = [n][n-1]\cdots [1],\end{equation} with the convention that $[0]! = 1$, and the binomial coefficient from the quantum factorial
\begin{equation} \label{e.qfac} \left[ \begin{array}{c} n \\ k \end{array} \right] = \frac{[n]!}{[k]![n-k]!}.  \end{equation}  
\begin{defn} \label{d.skein} Let $\mathbb{C}(q)$ be the field of rational functions in $q$ with complex coefficients. Given a connected, orientable surface $F$, the skein module $K(F)$ is the vector space over $\mathbb{C}(q)$ generated by the set of isotopy classes of framed unoriented and properly embedded link/tangle diagrams  $T$  modulo the Kauffman skein relations\footnote{Our convention differs from \cite{Kho97} in that $q$ there is our $q^{-1}$.} \cite{Kau87}: 
\begin{enumerate}
\item $\TikCircle  \sqcup T = (-q-q^{-1}) T $ 
\item $\vcenter{\hbox{\ocrossing}} = q^{-1/2} \  \vcenter{\hbox{\ares}} + q^{1/2} \ \vcenter{\hbox{\bres}}$
\end{enumerate} 
\end{defn} 
Let $D^2$ be a disk in the plane whose boundary is viewed as a rectangle, with $n$ marked points above and below. The \emph{Temperley-Lieb algebra} $TL_n$ is the skein module $K(D^2)$ of tangle/link diagrams in this disk, where we restrict to tangle diagrams whose endpoints are a subset of the $2n$ marked points on the boundary of the disk. There is a natural multiplication of two elements $T_1, T_2 \mapsto T_1 \circ T_2$  in $TL_n$ by stacking the square containing $T_1$ above the square containing $T_2$ and identifying the $n$-boundary points. See Figure \ref{f.tlstack} below. 
\begin{figure}[H]
\def \svgwidth{.2\textwidth}
\begingroup%
  \makeatletter%
  \providecommand\color[2][]{%
    \errmessage{(Inkscape) Color is used for the text in Inkscape, but the package 'color.sty' is not loaded}%
    \renewcommand\color[2][]{}%
  }%
  \providecommand\transparent[1]{%
    \errmessage{(Inkscape) Transparency is used (non-zero) for the text in Inkscape, but the package 'transparent.sty' is not loaded}%
    \renewcommand\transparent[1]{}%
  }%
  \providecommand\rotatebox[2]{#2}%
  \newcommand*\fsize{\dimexpr\f@size pt\relax}%
  \newcommand*\lineheight[1]{\fontsize{\fsize}{#1\fsize}\selectfont}%
  \ifx\svgwidth\undefined%
    \setlength{\unitlength}{659.07463074bp}%
    \ifx\svgscale\undefined%
      \relax%
    \else%
      \setlength{\unitlength}{\unitlength * \real{\svgscale}}%
    \fi%
  \else%
    \setlength{\unitlength}{\svgwidth}%
  \fi%
  \global\let\svgwidth\undefined%
  \global\let\svgscale\undefined%
  \makeatother%
  \begin{picture}(1,0.39511582)%
    \lineheight{1}%
    \setlength\tabcolsep{0pt}%
    \put(0,0){\includegraphics[width=\unitlength,page=1]{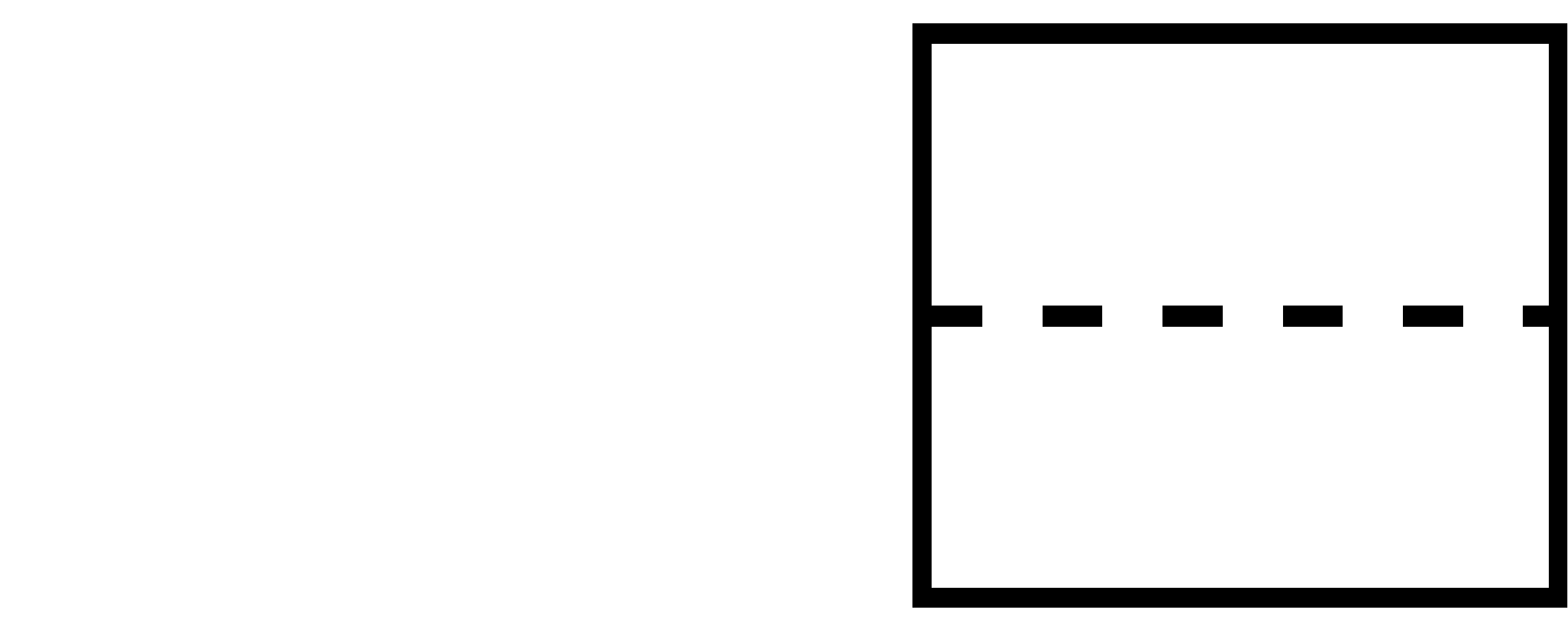}}%
    \put(0.71329426,0.26186113){\makebox(0,0)[lt]{\lineheight{1.25}\smash{\begin{tabular}[t]{l}$T_1$\end{tabular}}}}%
    \put(0.72033586,0.07735974){\makebox(0,0)[lt]{\lineheight{1.25}\smash{\begin{tabular}[t]{l}$T_2$\end{tabular}}}}%
    \put(-0.00229667,0.18383465){\makebox(0,0)[lt]{\lineheight{1.25}\smash{\begin{tabular}[t]{l}$T_1 \circ T_2=$\end{tabular}}}}%
    \put(0,0){\includegraphics[width=\unitlength,page=2]{tlstack.pdf}}%
  \end{picture}%
\endgroup%
 
\caption{\label{f.tlstack} $T_1 \circ T_2$ in $TL_n$.}
\end{figure} 

In this paper the rectangles $D^2$ with $2n$ marked boundary points defining $TL_n$ will sometimes be depicted with a rotation from the horizontal position of Figure \ref{f.tlstack}. One can infer the direction of the multiplication based on the context of the figures. An algebra generator $U_i$ of $TL_n$ is shown below. 
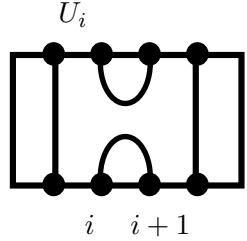
\begin{figure}[H] 
\def \svgwidth{.2\textwidth}
\begingroup%
  \makeatletter%
  \providecommand\color[2][]{%
    \errmessage{(Inkscape) Color is used for the text in Inkscape, but the package 'color.sty' is not loaded}%
    \renewcommand\color[2][]{}%
  }%
  \providecommand\transparent[1]{%
    \errmessage{(Inkscape) Transparency is used (non-zero) for the text in Inkscape, but the package 'transparent.sty' is not loaded}%
    \renewcommand\transparent[1]{}%
  }%
  \providecommand\rotatebox[2]{#2}%
  \newcommand*\fsize{\dimexpr\f@size pt\relax}%
  \newcommand*\lineheight[1]{\fontsize{\fsize}{#1\fsize}\selectfont}%
  \ifx\svgwidth\undefined%
    \setlength{\unitlength}{114.08927147bp}%
    \ifx\svgscale\undefined%
      \relax%
    \else%
      \setlength{\unitlength}{\unitlength * \real{\svgscale}}%
    \fi%
  \else%
    \setlength{\unitlength}{\svgwidth}%
  \fi%
  \global\let\svgwidth\undefined%
  \global\let\svgscale\undefined%
  \makeatother%
  \begin{picture}(1,1.02039409)%
    \lineheight{1}%
    \setlength\tabcolsep{0pt}%
    \put(0,0){\includegraphics[width=\unitlength,page=1]{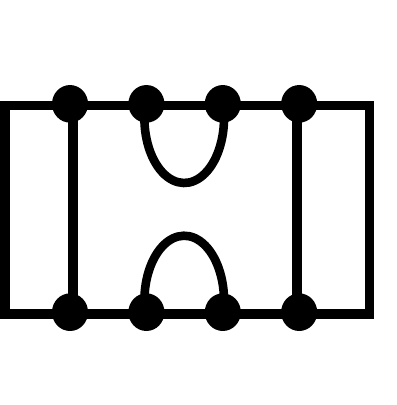}}%
    \put(0.30167024,0.0309003){\makebox(0,0)[lt]{\lineheight{1.25}\smash{\begin{tabular}[t]{l}$i$\end{tabular}}}}%
    \put(0.48573681,0.0309003){\makebox(0,0)[lt]{\lineheight{1.25}\smash{\begin{tabular}[t]{l}$i+1$\end{tabular}}}}%
    \put(0.1964893,0.88549416){\makebox(0,0)[lt]{\lineheight{1.25}\smash{\begin{tabular}[t]{l}$U_i$\end{tabular}}}}%
  \end{picture}%
\endgroup%

\caption{An algebra generator $U_i$ of $TL_n$}
\end{figure} 

The canonical dual basis $B^{TL}_n$ of $TL_n$ is the set of all crossingless matchings of the $2n$ boundary points of $D^2$.

\begin{defn}  \label{d.jw} The Jones-Wenzl projector $\vcenter{\hbox{\tiny{\def \svgwidth{.025\columnwidth}}}}_n$ is an element in $TL_n$ that is uniquely characterized by the following properties: 
\begin{itemize}
\item $\vcenter{\hbox{\tiny{\def \svgwidth{.025\columnwidth}}}}_n \circ \vcenter{\hbox{\tiny{\def \svgwidth{.025\columnwidth}}}}_n= \vcenter{\hbox{\tiny{\def \svgwidth{.025\columnwidth}}}}_n, n\geq 1.$
\item $\vcenter{\hbox{\tiny{\def \svgwidth{.025\columnwidth}}}}_n \circ U_i = 0 = U_i \circ \vcenter{\hbox{\tiny{\def \svgwidth{.025\columnwidth}}}}_n, 1\leq i \leq n-1.$ 
\end{itemize} 
\end{defn} 

\begin{defn} \label{d.dex}
Let $d$ be an element in the canonical dual basis $B^{TL}_n$ of $TL_n$. We define $P(d)$ to be the coefficient in the expansion of the Jones-Wenzl projector $JW_n$ in $TL_n$.  
\[ [n]!\vcenter{\hbox{\tiny{\def \svgwidth{.025\columnwidth}}}}_n = \sum_{d\in B^{TL}_n} P(d) d. \] 
$P(d)$ is a rational function in $q$ with complex coefficients. We will denote by $\P(d)$ the coefficient divided by $[n]!$: 
\[ \P(d) = \frac{P(d)}{[n]!}. \] 
\end{defn} 
Henceforth, whenever we say ``expansion of the Jones-Wenzl projector," we shall mean the expansion of the Jones-Wenzl projector in terms of the canonical basis as in Definition \ref{d.dex}. 

\begin{defn} \label{d.kbracket}
Consider the skein module $K(\mathbb{R}^2)$ of the plane $\mathbb{R}$.  The \emph{Kauffman bracket} of a skein element $\sk$ in $K(\mathbb{R}^2)$, denoted by $\langle \sk \rangle$, is the rational function in $q$ multiplying the empty diagram after resolving $\sk$ via the Kauffman skein relations.
\end{defn} 

\begin{defn} \label{d.cjp}
Given a link diagram $L$ and an integer $n\geq 2$, cable a Jones-Wenzl projector in each component and denote the resulting diagram by $L^n$. Then the $n+1$ colored Jones polynomial $J_{L, n+1}$ may be defined as 
\[ J_{L, n+1} = ((-1)^n q^{1/2})^{\omega(L)(n^2+2n)}  \langle L^n \rangle. \]
\end{defn} 
 This material is well known and interested readers may consult \cite{Oh01}, \cite{Lic97} for additional background. 
\paragraph{\textbf{Normalization}}
With our conventions, we have that the $n$ \emph{(unreduced) colored Jones polynomial} of the unknot  is
 \[ J_{\Circle, n} = (-1)^{n-1} [n].\] 

We will be using the following lemmas from \cite{Kho97}. All symbols $x, y, z, t, a, b, c, k, n$ will denote non-negative integers.  A non-negative integer next to a strand indicates that number of parallel strands. 
\begin{lem}{\cite[Proposition 4.8]{Kho97}}  \label{l.ccup}
Let $d_1$ and $d_2$ be two diagrams in $TL_n$ that differ as shown below
\begin{figure}[H]
\def \svgwidth{.4\columnwidth}
\begingroup%
  \makeatletter%
  \providecommand\color[2][]{%
    \errmessage{(Inkscape) Color is used for the text in Inkscape, but the package 'color.sty' is not loaded}%
    \renewcommand\color[2][]{}%
  }%
  \providecommand\transparent[1]{%
    \errmessage{(Inkscape) Transparency is used (non-zero) for the text in Inkscape, but the package 'transparent.sty' is not loaded}%
    \renewcommand\transparent[1]{}%
  }%
  \providecommand\rotatebox[2]{#2}%
  \newcommand*\fsize{\dimexpr\f@size pt\relax}%
  \newcommand*\lineheight[1]{\fontsize{\fsize}{#1\fsize}\selectfont}%
  \ifx\svgwidth\undefined%
    \setlength{\unitlength}{1069.28573488bp}%
    \ifx\svgscale\undefined%
      \relax%
    \else%
      \setlength{\unitlength}{\unitlength * \real{\svgscale}}%
    \fi%
  \else%
    \setlength{\unitlength}{\svgwidth}%
  \fi%
  \global\let\svgwidth\undefined%
  \global\let\svgscale\undefined%
  \makeatother%
  \begin{picture}(1,0.24213923)%
    \lineheight{1}%
    \setlength\tabcolsep{0pt}%
    \put(0,0){\includegraphics[width=\unitlength,page=1]{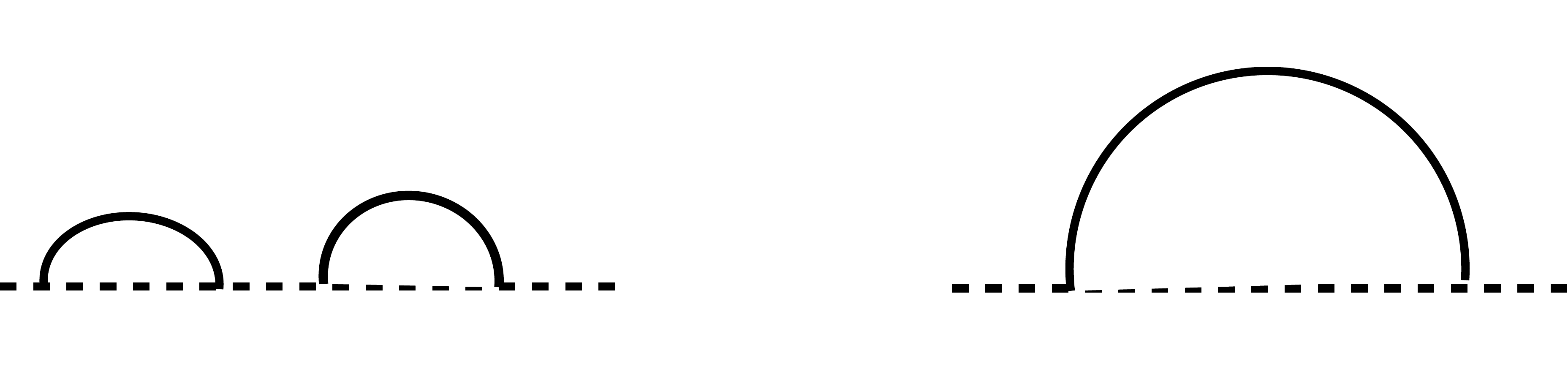}}%
    \put(0.05611222,0.14357754){\makebox(0,0)[lt]{\lineheight{1.25}\smash{\begin{tabular}[t]{l}$x$\end{tabular}}}}%
    \put(0.22444889,0.14357754){\makebox(0,0)[lt]{\lineheight{1.25}\smash{\begin{tabular}[t]{l}$y$\end{tabular}}}}%
    \put(0.7855711,0.22774585){\makebox(0,0)[lt]{\lineheight{1.25}\smash{\begin{tabular}[t]{l}$x+y$\end{tabular}}}}%
    \put(0.14028056,0.00329697){\makebox(0,0)[lt]{\lineheight{1.25}\smash{\begin{tabular}[t]{l}$d_1$\end{tabular}}}}%
    \put(0.75751504,0.00329697){\makebox(0,0)[lt]{\lineheight{1.25}\smash{\begin{tabular}[t]{l}$d_2$\end{tabular}}}}%
  \end{picture}%
\endgroup%

\caption{}
\end{figure} 
They are the same outside the part that is shown. Then 
\[ P(d_1)  = \left[ \begin{array}{c} x+y\\x\end{array} \right] P(d_2) . \]

\end{lem} 
\begin{lem}{\cite[Proposition 4.9]{Kho97}}  \label{l.Khslide}
Let $d_1, d_2$ be two diagrams from $TL_n$ that differ as depicted below. 
\begin{figure}[H]
\def \svgwidth{.4\columnwidth}
\begingroup%
  \makeatletter%
  \providecommand\color[2][]{%
    \errmessage{(Inkscape) Color is used for the text in Inkscape, but the package 'color.sty' is not loaded}%
    \renewcommand\color[2][]{}%
  }%
  \providecommand\transparent[1]{%
    \errmessage{(Inkscape) Transparency is used (non-zero) for the text in Inkscape, but the package 'transparent.sty' is not loaded}%
    \renewcommand\transparent[1]{}%
  }%
  \providecommand\rotatebox[2]{#2}%
  \newcommand*\fsize{\dimexpr\f@size pt\relax}%
  \newcommand*\lineheight[1]{\fontsize{\fsize}{#1\fsize}\selectfont}%
  \ifx\svgwidth\undefined%
    \setlength{\unitlength}{407.83460242bp}%
    \ifx\svgscale\undefined%
      \relax%
    \else%
      \setlength{\unitlength}{\unitlength * \real{\svgscale}}%
    \fi%
  \else%
    \setlength{\unitlength}{\svgwidth}%
  \fi%
  \global\let\svgwidth\undefined%
  \global\let\svgscale\undefined%
  \makeatother%
  \begin{picture}(1,0.40262773)%
    \lineheight{1}%
    \setlength\tabcolsep{0pt}%
    \put(0,0){\includegraphics[width=\unitlength,page=1]{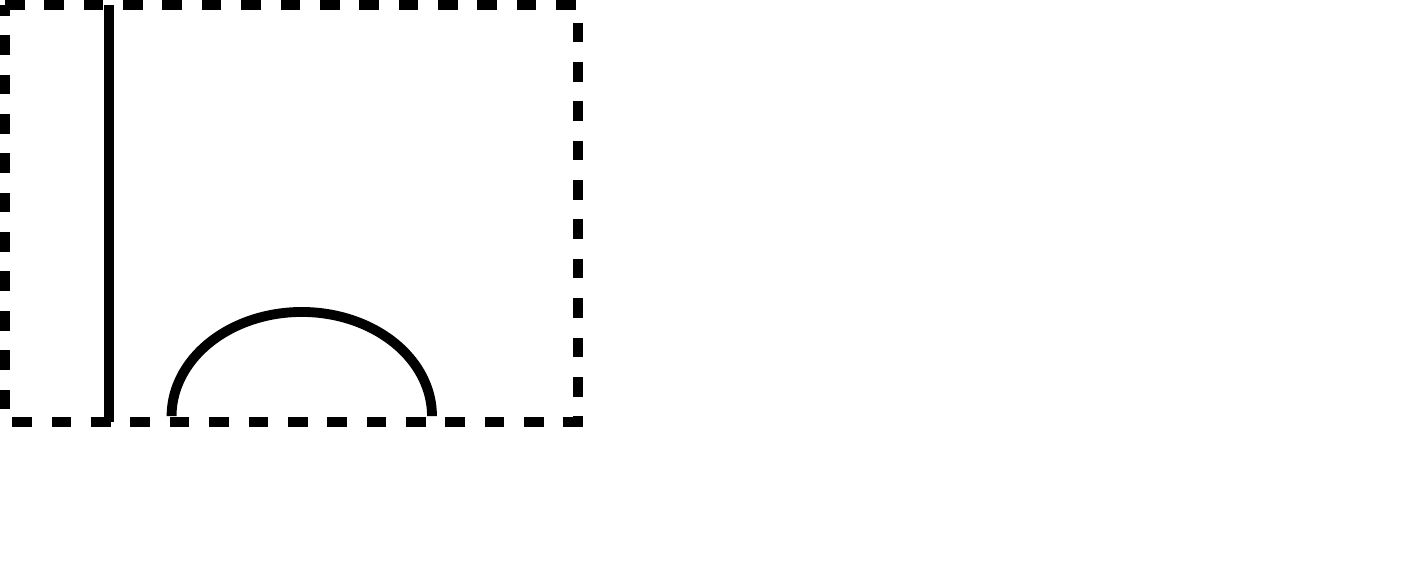}}%
    \put(0.0807124,0.01296626){\color[rgb]{0,0,0}\makebox(0,0)[lt]{\lineheight{1.25}\smash{\begin{tabular}[t]{l}$d_1$\end{tabular}}}}%
    \put(0.08491582,0.30194919){\color[rgb]{0,0,0}\makebox(0,0)[lt]{\lineheight{1.25}\smash{\begin{tabular}[t]{l}$x$\end{tabular}}}}%
    \put(0.26093253,0.21525439){\color[rgb]{0,0,0}\makebox(0,0)[lt]{\lineheight{1.25}\smash{\begin{tabular}[t]{l}$y$\end{tabular}}}}%
    \put(0,0){\includegraphics[width=\unitlength,page=2]{mstrand.pdf}}%
    \put(0.66918624,0.01296626){\color[rgb]{0,0,0}\makebox(0,0)[lt]{\lineheight{1.25}\smash{\begin{tabular}[t]{l}$d_2$\end{tabular}}}}%
    \put(0,0){\includegraphics[width=\unitlength,page=3]{mstrand.pdf}}%
    \put(0.68915232,0.29669497){\color[rgb]{0,0,0}\makebox(0,0)[lt]{\lineheight{1.25}\smash{\begin{tabular}[t]{l}$x$\end{tabular}}}}%
    \put(0.61033874,0.13801718){\color[rgb]{0,0,0}\makebox(0,0)[lt]{\lineheight{1.25}\smash{\begin{tabular}[t]{l}$y$\end{tabular}}}}%
  \end{picture}%
\endgroup%

\caption{}
\end{figure} 
They are the same outside the part that is shown. Then 
\[ P(d_1) = \left[ \begin{array}{c}x+y \\ x \end{array} \right] P(d_2). \] 
\end{lem} 
\begin{lem}{\cite[Proposition 4.10]{Kho97}} \label{l.khprop}
Let $d$ be a diagram
as shown, where $x, y, z, t$ are non-negative integers satisfying $x+y+z+t\geq 1$. 
\begin{figure}[H]
\def \svgwidth{.2\columnwidth}
\begingroup%
  \makeatletter%
  \providecommand\color[2][]{%
    \errmessage{(Inkscape) Color is used for the text in Inkscape, but the package 'color.sty' is not loaded}%
    \renewcommand\color[2][]{}%
  }%
  \providecommand\transparent[1]{%
    \errmessage{(Inkscape) Transparency is used (non-zero) for the text in Inkscape, but the package 'transparent.sty' is not loaded}%
    \renewcommand\transparent[1]{}%
  }%
  \providecommand\rotatebox[2]{#2}%
  \newcommand*\fsize{\dimexpr\f@size pt\relax}%
  \newcommand*\lineheight[1]{\fontsize{\fsize}{#1\fsize}\selectfont}%
  \ifx\svgwidth\undefined%
    \setlength{\unitlength}{508.22174313bp}%
    \ifx\svgscale\undefined%
      \relax%
    \else%
      \setlength{\unitlength}{\unitlength * \real{\svgscale}}%
    \fi%
  \else%
    \setlength{\unitlength}{\svgwidth}%
  \fi%
  \global\let\svgwidth\undefined%
  \global\let\svgscale\undefined%
  \makeatother%
  \begin{picture}(1,0.52774439)%
    \lineheight{1}%
    \setlength\tabcolsep{0pt}%
    \put(0,0){\includegraphics[width=\unitlength,page=1]{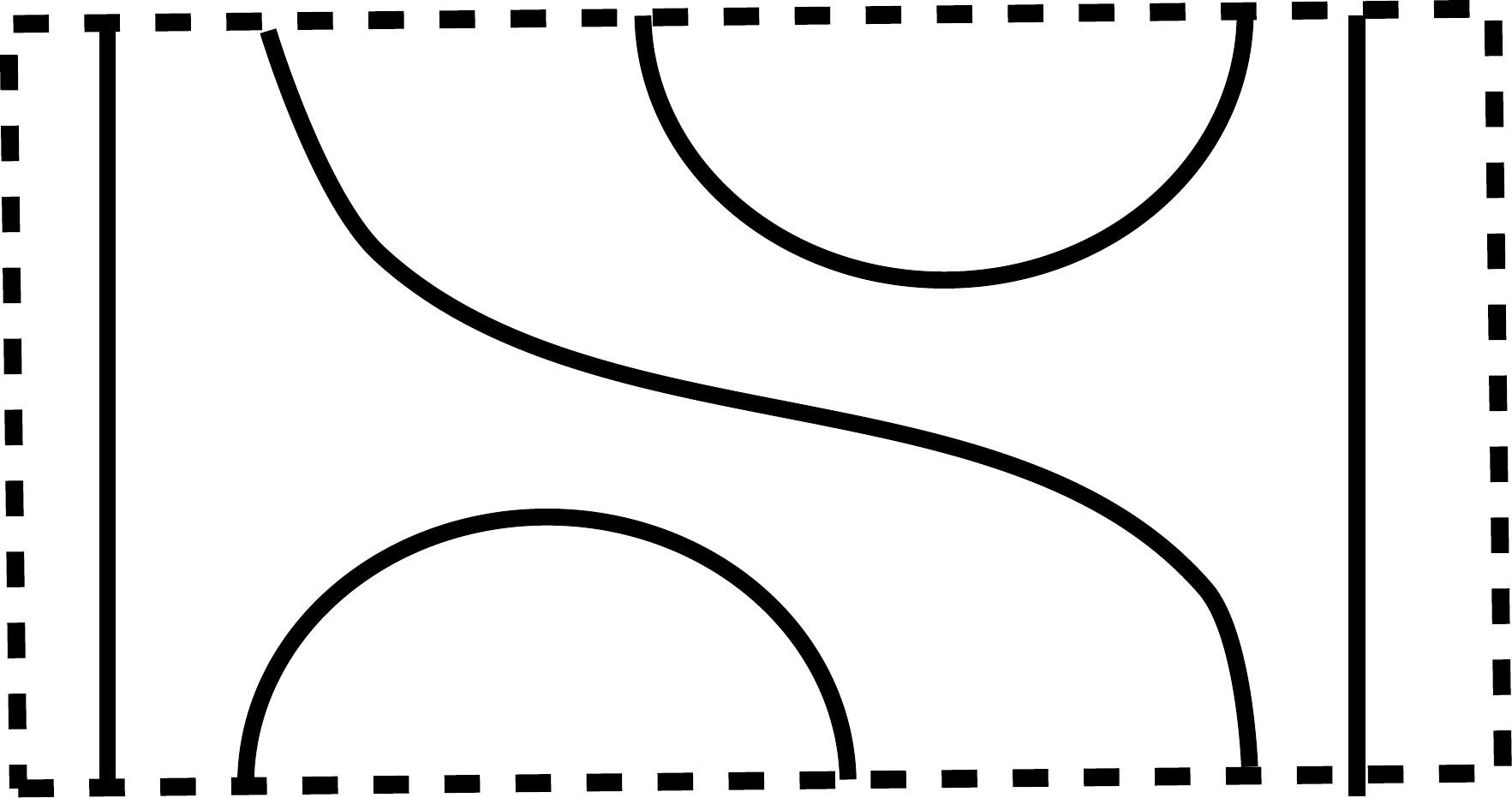}}%
    \put(0.08881332,0.22220897){\makebox(0,0)[lt]{\lineheight{1.25}\smash{\begin{tabular}[t]{l}$x$\end{tabular}}}}%
    \put(0.72042738,0.22220897){\makebox(0,0)[lt]{\lineheight{1.25}\smash{\begin{tabular}[t]{l}$z$\end{tabular}}}}%
    \put(0.66139806,0.39929701){\makebox(0,0)[lt]{\lineheight{1.25}\smash{\begin{tabular}[t]{l}$y$\end{tabular}}}}%
    \put(0.33673663,0.10415024){\makebox(0,0)[lt]{\lineheight{1.25}\smash{\begin{tabular}[t]{l}$y$\end{tabular}}}}%
    \put(0.90932157,0.25172365){\makebox(0,0)[lt]{\lineheight{1.25}\smash{\begin{tabular}[t]{l}$t$\end{tabular}}}}%
  \end{picture}%
\endgroup%

\end{figure} 
Then 
\[P(d) = \left[ \begin{array}{c} x+y \\ y \end{array}\right] \left[ \begin{array}{c} t+y \\ y \end{array}\right] [y]! [x+z+t+y]!.  \] 
\end{lem} 

We will also take as given the following fusion and untwisting formulas, and the formula for removing circles from a projector  \cite{MV94}.
\begin{figure}[H]
\begin{center} 
\def \svgwidth{.7\columnwidth}
\begin{equation} \label{e.funtwist}
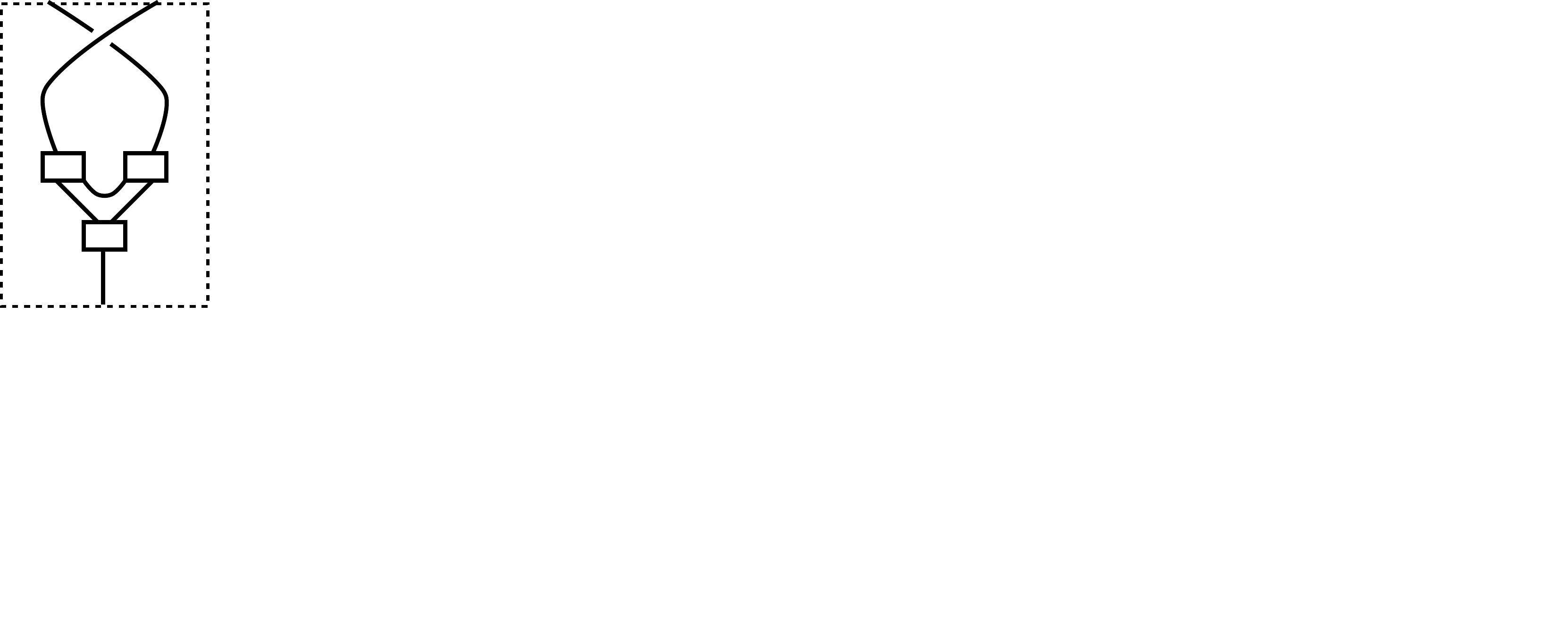
\end{equation} 
\end{center} 
\caption{Fusion and untwisting formulas}
\end{figure} 
\begin{figure}[H]
 \begin{center} 
\def \svgwidth{.35\columnwidth}
\begin{equation} \label{e.rcircle}
\begingroup%
  \makeatletter%
  \providecommand\color[2][]{%
    \errmessage{(Inkscape) Color is used for the text in Inkscape, but the package 'color.sty' is not loaded}%
    \renewcommand\color[2][]{}%
  }%
  \providecommand\transparent[1]{%
    \errmessage{(Inkscape) Transparency is used (non-zero) for the text in Inkscape, but the package 'transparent.sty' is not loaded}%
    \renewcommand\transparent[1]{}%
  }%
  \providecommand\rotatebox[2]{#2}%
  \newcommand*\fsize{\dimexpr\f@size pt\relax}%
  \newcommand*\lineheight[1]{\fontsize{\fsize}{#1\fsize}\selectfont}%
  \ifx\svgwidth\undefined%
    \setlength{\unitlength}{1309.22532112bp}%
    \ifx\svgscale\undefined%
      \relax%
    \else%
      \setlength{\unitlength}{\unitlength * \real{\svgscale}}%
    \fi%
  \else%
    \setlength{\unitlength}{\svgwidth}%
  \fi%
  \global\let\svgwidth\undefined%
  \global\let\svgscale\undefined%
  \makeatother%
  \begin{picture}(1,0.32080039)%
    \lineheight{1}%
    \setlength\tabcolsep{0pt}%
    \put(0,0){\includegraphics[width=\unitlength,page=1]{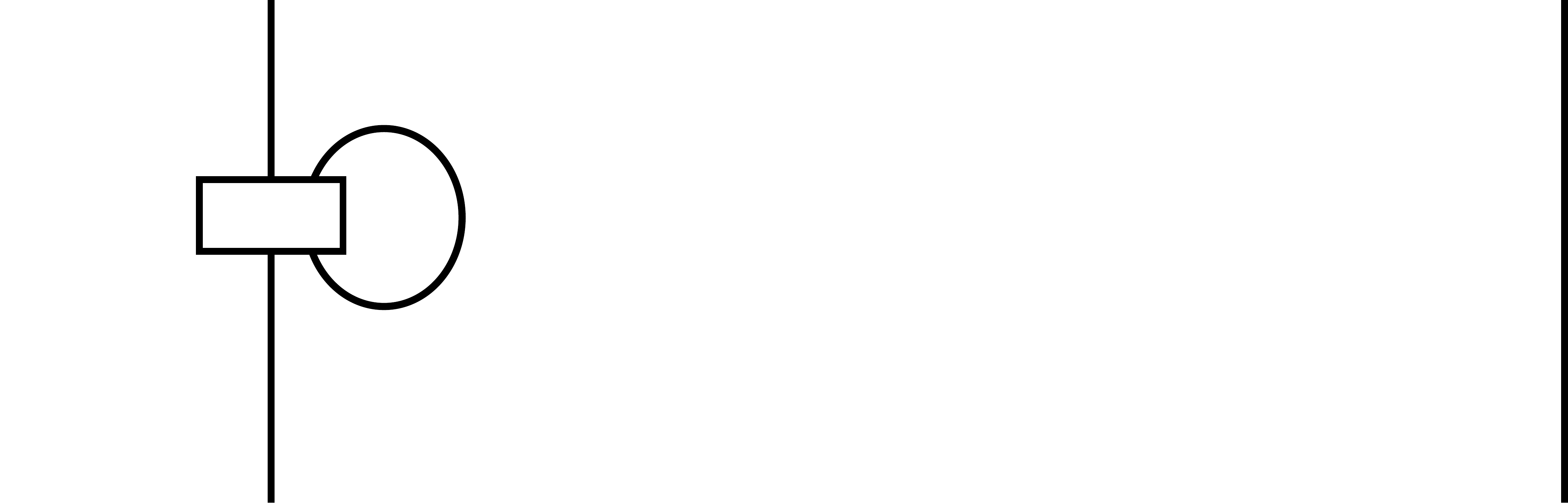}}%
    \put(-0.00115616,0.23077987){\makebox(0,0)[lt]{\lineheight{1.25}\smash{\begin{tabular}[t]{l}$n-c$\end{tabular}}}}%
    \put(0.24166251,0.25205745){\makebox(0,0)[lt]{\lineheight{1.25}\smash{\begin{tabular}[t]{l}$c$\end{tabular}}}}%
    \put(0.31040545,0.13748589){\makebox(0,0)[lt]{\lineheight{1.25}\smash{\begin{tabular}[t]{l}$=(-1)^c\frac{[n+2]}{[n+2-c]}$\end{tabular}}}}%
    \put(0.83268554,0.22914313){\makebox(0,0)[lt]{\lineheight{1.25}\smash{\begin{tabular}[t]{l}$n-c$\end{tabular}}}}%
  \end{picture}%
\endgroup%

\end{equation} 
\caption{The formula for removing circles from a Jones-Wenzl projector}
\end{center} 
\end{figure} 
Define 
\[U(w, k) = ((-1)^{n-k} q^{n-k+ \frac{n^2}{2}-k^2})^w. \] 
We have from \eqref{e.funtwist},  \\  
\begin{equation*}
\def \svgwidth{.9\columnwidth}
\begingroup%
  \makeatletter%
  \providecommand\color[2][]{%
    \errmessage{(Inkscape) Color is used for the text in Inkscape, but the package 'color.sty' is not loaded}%
    \renewcommand\color[2][]{}%
  }%
  \providecommand\transparent[1]{%
    \errmessage{(Inkscape) Transparency is used (non-zero) for the text in Inkscape, but the package 'transparent.sty' is not loaded}%
    \renewcommand\transparent[1]{}%
  }%
  \providecommand\rotatebox[2]{#2}%
  \newcommand*\fsize{\dimexpr\f@size pt\relax}%
  \newcommand*\lineheight[1]{\fontsize{\fsize}{#1\fsize}\selectfont}%
  \ifx\svgwidth\undefined%
    \setlength{\unitlength}{864.05792521bp}%
    \ifx\svgscale\undefined%
      \relax%
    \else%
      \setlength{\unitlength}{\unitlength * \real{\svgscale}}%
    \fi%
  \else%
    \setlength{\unitlength}{\svgwidth}%
  \fi%
  \global\let\svgwidth\undefined%
  \global\let\svgscale\undefined%
  \makeatother%
  \begin{picture}(1,0.1909843)%
    \lineheight{1}%
    \setlength\tabcolsep{0pt}%
    \put(0,0){\includegraphics[width=\unitlength,page=1]{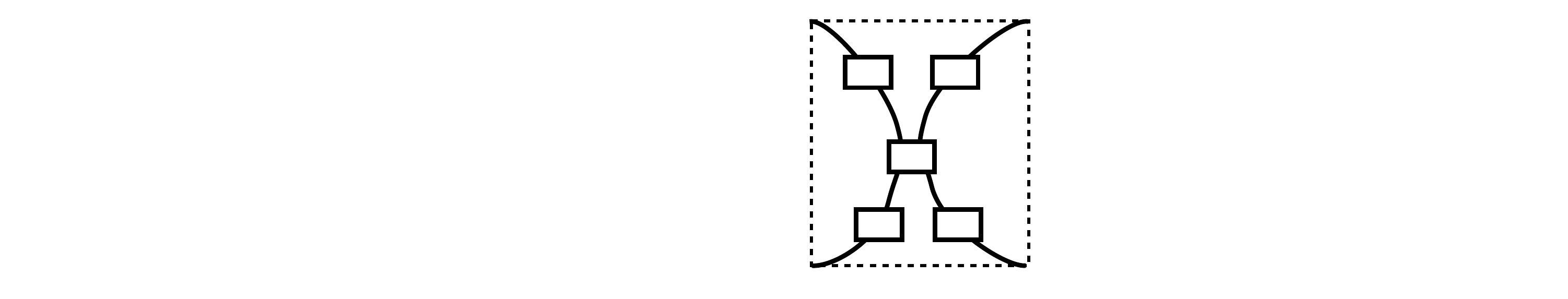}}%
    \put(0.54725863,0.15987769){\makebox(0,0)[lt]{\lineheight{1.25}\smash{\begin{tabular}[t]{l}$n$\end{tabular}}}}%
    \put(0.60281016,0.15987779){\makebox(0,0)[lt]{\lineheight{1.25}\smash{\begin{tabular}[t]{l}$n$\end{tabular}}}}%
    \put(0.52121906,0.02967809){\makebox(0,0)[lt]{\lineheight{1.25}\smash{\begin{tabular}[t]{l}$n$\end{tabular}}}}%
    \put(0.634058,0.03835807){\makebox(0,0)[lt]{\lineheight{1.25}\smash{\begin{tabular}[t]{l}$n$\end{tabular}}}}%
    \put(0.54725863,0.10779785){\makebox(0,0)[lt]{\lineheight{1.25}\smash{\begin{tabular}[t]{l}$k$\end{tabular}}}}%
    \put(0.59933812,0.10779785){\makebox(0,0)[lt]{\lineheight{1.25}\smash{\begin{tabular}[t]{l}$k$\end{tabular}}}}%
    \put(0.54725863,0.06787051){\makebox(0,0)[lt]{\lineheight{1.25}\smash{\begin{tabular}[t]{l}$k$\end{tabular}}}}%
    \put(0.59933812,0.06787051){\makebox(0,0)[lt]{\lineheight{1.25}\smash{\begin{tabular}[t]{l}$k$\end{tabular}}}}%
    \put(0.10325273,0.10057441){\makebox(0,0)[lt]{\lineheight{1.25}\smash{\begin{tabular}[t]{l}$=\sum_{\substack{k\leq n:(n, n, k) \\ \text{admissible}}}(-1)^{k} \frac{[k+1]}{\theta(n, n, k)}U(w, k)$\end{tabular}}}}%
    \put(0,0){\includegraphics[width=\unitlength,page=2]{untwist.pdf}}%
    \put(0.03024009,0.00478077){\makebox(0,0)[lt]{\lineheight{1.25}\smash{\begin{tabular}[t]{l}$|w|$ crossings\end{tabular}}}}%
    \put(0,0){\includegraphics[width=\unitlength,page=3]{untwist.pdf}}%
    \put(0.061488,0.11762044){\makebox(0,0)[lt]{\lineheight{1.25}\smash{\begin{tabular}[t]{l}.\end{tabular}}}}%
    \put(0.061488,0.10199647){\makebox(0,0)[lt]{\lineheight{1.25}\smash{\begin{tabular}[t]{l}.\end{tabular}}}}%
    \put(0.061488,0.11067645){\makebox(0,0)[lt]{\lineheight{1.25}\smash{\begin{tabular}[t]{l}.\end{tabular}}}}%
    \put(-0.00100782,0.17317226){\makebox(0,0)[lt]{\lineheight{1.25}\smash{\begin{tabular}[t]{l}$n$\end{tabular}}}}%
    \put(0.10662385,0.17317226){\makebox(0,0)[lt]{\lineheight{1.25}\smash{\begin{tabular}[t]{l}$n$\end{tabular}}}}%
    \put(0.10823584,0.02437271){\makebox(0,0)[lt]{\lineheight{1.25}\smash{\begin{tabular}[t]{l}$n$\end{tabular}}}}%
    \put(-0.00175182,0.02065272){\makebox(0,0)[lt]{\lineheight{1.25}\smash{\begin{tabular}[t]{l}$n$\end{tabular}}}}%
  \end{picture}%
\endgroup%

\end{equation*} 
Also define 
\[ R(c):= (-1)^c \frac{[n+2]}{[n+2-c]}. \]  
This is the coefficient multiplying $n-c$ strands after removing $c$ circles using \eqref{e.rcircle}.

A triple of even integers $(a, b, c)$ is called admissible if $a\leq b+c$, $b\leq a+c$ and $c\leq a+b$. For an admissible triple of integers $(a, b, c)$, the function $\theta(a, b, c)$ is $\langle \Theta (a, b, c)  \rangle$ of the skein element $\Theta(a, b, c)$ as shown below in Figure \ref{f.theta}: 
\[\theta(a, b, c) = \frac{[x+y+z+1]![x]![y]![z]!}{[y+z]![z+x]![x+y]!},  \]
where $x = \frac{a+b-c}{2}$, $y = \frac{b+c-a}{2}$, and $z = \frac{a+c-b}{2}$.  
\begin{figure}[H]
\def \svgwidth{.3\columnwidth}
\begingroup%
  \makeatletter%
  \providecommand\color[2][]{%
    \errmessage{(Inkscape) Color is used for the text in Inkscape, but the package 'color.sty' is not loaded}%
    \renewcommand\color[2][]{}%
  }%
  \providecommand\transparent[1]{%
    \errmessage{(Inkscape) Transparency is used (non-zero) for the text in Inkscape, but the package 'transparent.sty' is not loaded}%
    \renewcommand\transparent[1]{}%
  }%
  \providecommand\rotatebox[2]{#2}%
  \newcommand*\fsize{\dimexpr\f@size pt\relax}%
  \newcommand*\lineheight[1]{\fontsize{\fsize}{#1\fsize}\selectfont}%
  \ifx\svgwidth\undefined%
    \setlength{\unitlength}{347.50484677bp}%
    \ifx\svgscale\undefined%
      \relax%
    \else%
      \setlength{\unitlength}{\unitlength * \real{\svgscale}}%
    \fi%
  \else%
    \setlength{\unitlength}{\svgwidth}%
  \fi%
  \global\let\svgwidth\undefined%
  \global\let\svgscale\undefined%
  \makeatother%
  \begin{picture}(1,0.68457405)%
    \lineheight{1}%
    \setlength\tabcolsep{0pt}%
    \put(0,0){\includegraphics[width=\unitlength,page=1]{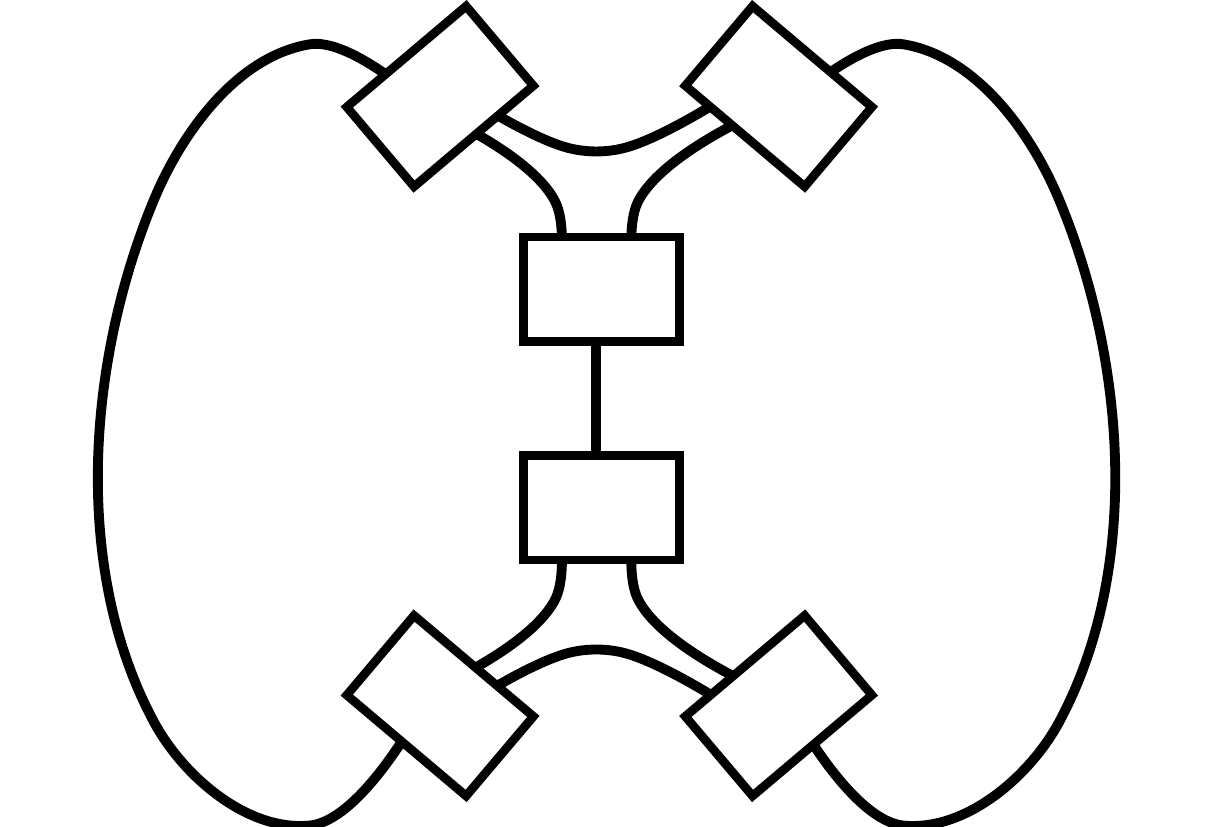}}%
    \put(-0.00217791,0.30254649){\makebox(0,0)[lt]{\lineheight{1.25}\smash{\begin{tabular}[t]{l}$a$\end{tabular}}}}%
    \put(0.42947075,0.34571133){\makebox(0,0)[lt]{\lineheight{1.25}\smash{\begin{tabular}[t]{l}$c$\end{tabular}}}}%
    \put(0.94744914,0.43204108){\makebox(0,0)[lt]{\lineheight{1.25}\smash{\begin{tabular}[t]{l}$b$\end{tabular}}}}%
    \put(0.4726356,0.60470052){\makebox(0,0)[lt]{\lineheight{1.25}\smash{\begin{tabular}[t]{l}$z$\end{tabular}}}}%
    \put(0.38630588,0.49678838){\makebox(0,0)[lt]{\lineheight{1.25}\smash{\begin{tabular}[t]{l}$y$\end{tabular}}}}%
    \put(0.5805478,0.49678838){\makebox(0,0)[lt]{\lineheight{1.25}\smash{\begin{tabular}[t]{l}$x$\end{tabular}}}}%
    \put(0.5805478,0.17305195){\makebox(0,0)[lt]{\lineheight{1.25}\smash{\begin{tabular}[t]{l}$x$\end{tabular}}}}%
    \put(0.37540512,0.17741225){\makebox(0,0)[lt]{\lineheight{1.25}\smash{\begin{tabular}[t]{l}$y$\end{tabular}}}}%
    \put(0.47459665,0.09566193){\makebox(0,0)[lt]{\lineheight{1.25}\smash{\begin{tabular}[t]{l}$z$\end{tabular}}}}%
  \end{picture}%
\endgroup%

\caption{\label{f.theta} The skein element $\Theta(a, b, c)$}
\end{figure}
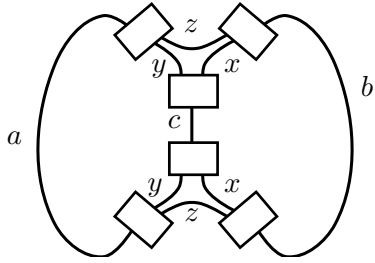 
For a rational function $\mathcal{L}(q)$ in $q$  complex coefficients, the degree, $\deg(\mathcal{L})$ is the maximum power in $q$ of the Laurent series expansion of $\mathcal{L}(q)$ with finitely many terms of non-negative powers. 
Note
\[\qquad \deg [c] = c-1, \qquad  \text{and} \qquad \deg \theta(a, b, c) = \frac{a+b+c}{2}.\]  

\section{Choosing skein elements in the expansion of a Jones-Wenzl projector} \label{s.org}
We collect useful lemmas classifying the possible choices of skein elements in the expansion of Jones-Wenzl projectors decorating certain skein elements. We will use these to justify our state sum of the colored Jones polynomial of pretzel links in Theorem \ref{t.mainintro}.  As assumed throughout the paper, the symbols labeling strands of tangle diagrams are non-negative integers, and they denote the number of parallel strands. We will use dashed bounding boxes to indicate the Jones-Wenzl projector we expand in each lemma. 
\begin{lem} \label{l.mjw}
Suppose that we have a skein element as in Figure \ref{f.jwc}. 
\begin{figure}[H]
\def\svgwidth{.2\columnwidth}
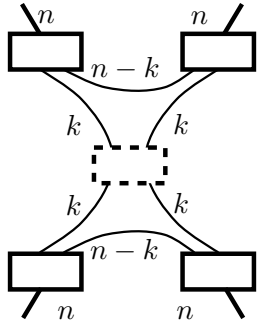
\caption{\label{f.jwc} The center Jones-Wenzl projector is shown with a dashed border. }
\end{figure} 
Then a skein element $d$  in the expansion of the center Jones-Wenzl projector (shown with a dashed bounding box) that does not result in a cap or a cup composed with one of the four framing Jones-Wenzl projectors is of the following form with $\tilde{k} \leq k$ :
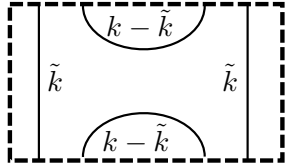
\begin{figure}[H]
\def\svgwidth{.3\columnwidth}
\begingroup%
  \makeatletter%
  \providecommand\color[2][]{%
    \errmessage{(Inkscape) Color is used for the text in Inkscape, but the package 'color.sty' is not loaded}%
    \renewcommand\color[2][]{}%
  }%
  \providecommand\transparent[1]{%
    \errmessage{(Inkscape) Transparency is used (non-zero) for the text in Inkscape, but the package 'transparent.sty' is not loaded}%
    \renewcommand\transparent[1]{}%
  }%
  \providecommand\rotatebox[2]{#2}%
  \newcommand*\fsize{\dimexpr\f@size pt\relax}%
  \newcommand*\lineheight[1]{\fontsize{\fsize}{#1\fsize}\selectfont}%
  \ifx\svgwidth\undefined%
    \setlength{\unitlength}{481.94392737bp}%
    \ifx\svgscale\undefined%
      \relax%
    \else%
      \setlength{\unitlength}{\unitlength * \real{\svgscale}}%
    \fi%
  \else%
    \setlength{\unitlength}{\svgwidth}%
  \fi%
  \global\let\svgwidth\undefined%
  \global\let\svgscale\undefined%
  \makeatother%
  \begin{picture}(1,0.43193675)%
    \lineheight{1}%
    \setlength\tabcolsep{0pt}%
    \put(0,0){\includegraphics[width=\unitlength,page=1]{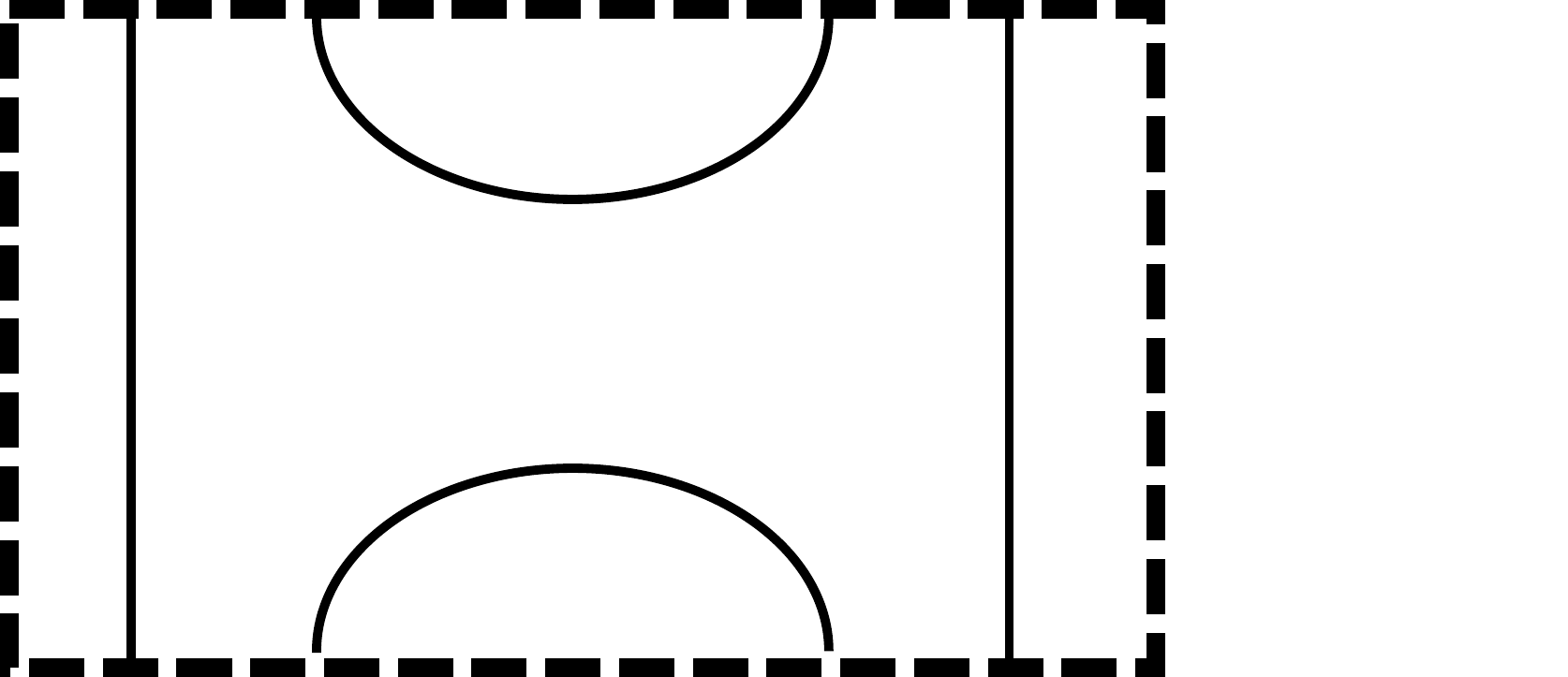}}%
    \put(0.26447916,0.3393854){\makebox(0,0)[lt]{\lineheight{1.25}\smash{\begin{tabular}[t]{l}$k-\tilde{k}$\end{tabular}}}}%
    \put(0.25380809,0.03229146){\makebox(0,0)[lt]{\lineheight{1.25}\smash{\begin{tabular}[t]{l}$k-\tilde{k}$\end{tabular}}}}%
    \put(0.10547835,0.19262541){\makebox(0,0)[lt]{\lineheight{1.25}\smash{\begin{tabular}[t]{l}$\tilde{k}$\end{tabular}}}}%
    \put(0.57544983,0.19262541){\makebox(0,0)[lt]{\lineheight{1.25}\smash{\begin{tabular}[t]{l}$\tilde{k}$\end{tabular}}}}%
  \end{picture}%
\endgroup%

\caption{Possibilities for the skein element $d$ in the expansion of the center Jones-Wenzl projector that does not result in a zero skein element.}
\end{figure} 
The choice of the skein element $d$ has coefficient
\[P(d) = \left[\begin{array}{c} k \\k-\tilde{k} \end{array}\right]^2  [k-\tilde{k}]![k+\tilde{k}]! \] in the expansion of the Jones-Wenzl projector. 
\begin{proof} Due to the four projectors framing the center projector, none of the four endpoints on the dashed bounding box, each with $k$ strands coming out, can be connected to itself through a strand from a choice of skein element in the expansion. Otherwise, the resulting skein element would have a cap or a cup composed with one of the framing projectors. Nor can any of the four endpoints be connected to another across the diagonal, because it would result in a cap or a cup composed with the pair of projectors of the opposite diagonal. 
\end{proof} 

\end{lem} 
\begin{defn}
Given a crossingless diagram in $TL_n$ without Jones-Wenzl projectors, a \emph{through strand} is a strand of the diagram with one end in one of the top $n$ points of the boundary of the disk $D^2$ defining $TL_n$,  and the other end in one of the bottom $n$ points. 
\end{defn}
Let $T\in TL_n$ be a skein element decorated by Jones-Wenzl projectors. We will denote by $\overline{T}$ the skein element obtained from $T$ by replacing all its Jones-Wenzl projectors by the identity.  
\begin{lem} \label{l.fmjw}
Suppose we have a skein element $T$ of the following form as in Figure \ref{f.cjw}.
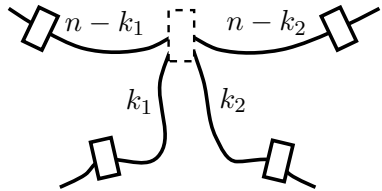
\begin{figure}[H]
\def \svgwidth{.3\columnwidth}
\begingroup%
  \makeatletter%
  \providecommand\color[2][]{%
    \errmessage{(Inkscape) Color is used for the text in Inkscape, but the package 'color.sty' is not loaded}%
    \renewcommand\color[2][]{}%
  }%
  \providecommand\transparent[1]{%
    \errmessage{(Inkscape) Transparency is used (non-zero) for the text in Inkscape, but the package 'transparent.sty' is not loaded}%
    \renewcommand\transparent[1]{}%
  }%
  \providecommand\rotatebox[2]{#2}%
  \newcommand*\fsize{\dimexpr\f@size pt\relax}%
  \newcommand*\lineheight[1]{\fontsize{\fsize}{#1\fsize}\selectfont}%
  \ifx\svgwidth\undefined%
    \setlength{\unitlength}{874.47152935bp}%
    \ifx\svgscale\undefined%
      \relax%
    \else%
      \setlength{\unitlength}{\unitlength * \real{\svgscale}}%
    \fi%
  \else%
    \setlength{\unitlength}{\svgwidth}%
  \fi%
  \global\let\svgwidth\undefined%
  \global\let\svgscale\undefined%
  \makeatother%
  \begin{picture}(1,0.49431367)%
    \lineheight{1}%
    \setlength\tabcolsep{0pt}%
    \put(0,0){\includegraphics[width=\unitlength,page=1]{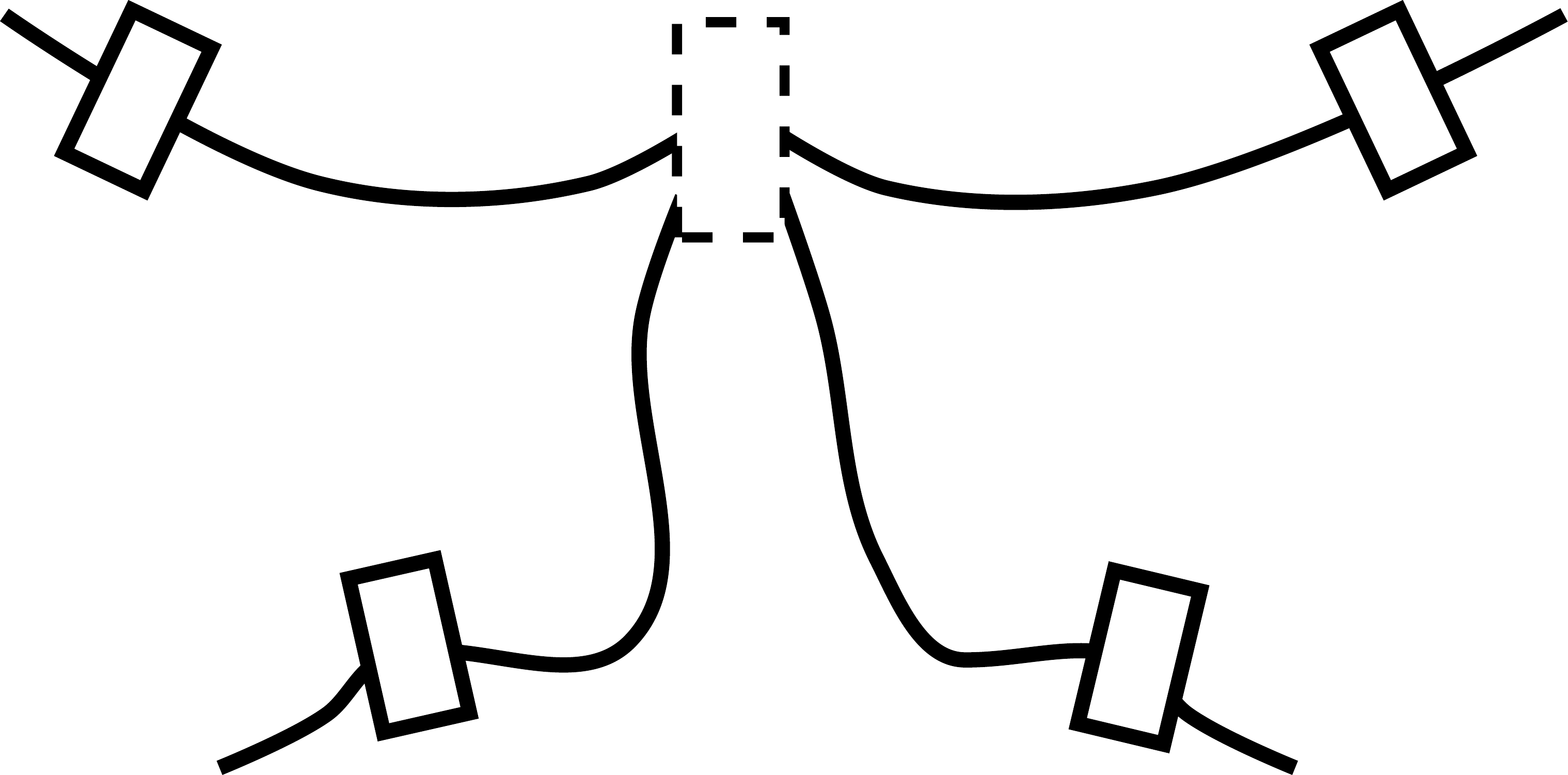}}%
    \put(0.15364456,0.42423311){\makebox(0,0)[lt]{\lineheight{1.25}\smash{\begin{tabular}[t]{l}$n-k_1$\end{tabular}}}}%
    \put(0.58762093,0.42594843){\makebox(0,0)[lt]{\lineheight{1.25}\smash{\begin{tabular}[t]{l}$n-k_2$\end{tabular}}}}%
    \put(0.31709023,0.21349358){\makebox(0,0)[lt]{\lineheight{1.25}\smash{\begin{tabular}[t]{l}$k_1$\end{tabular}}}}%
    \put(0.56850735,0.21814947){\makebox(0,0)[lt]{\lineheight{1.25}\smash{\begin{tabular}[t]{l}$k_2$\end{tabular}}}}%
  \end{picture}%
\endgroup%

\caption{\label{f.cjw} The center Jones-Wenzl projector is indicated with a dashed rectangular box.}
\end{figure} 

Denote by $T_{\sigma}$ the skein element obtained from $T$ by replacing the center Jones-Wenzl projector  (shown with a dashed bounding box) by a choice of a skein element $\sigma$ in its expansion. 
The choices of skein elements $\sigma = \sigma_t^1, \sigma_t^2$, or $\sigma_t^3$ which do not result in a cap or a cup composed with one of the four framing projectors have the following form.
\begin{figure}[H]
\def \svgwidth{.8\columnwidth}
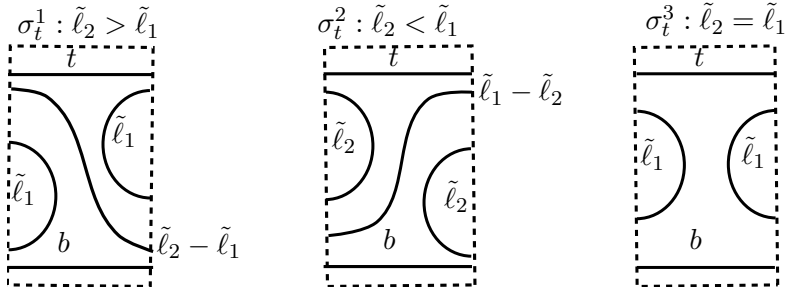
\caption{\label{f.glueskein-proj2} Three choices of skein elements}

\end{figure}
The coefficients of the choices of skein elements $\sigma_t^1, \sigma_t^2$, and $\sigma_t^3$ are respectively 
\[ P(\sigma_t^1) =  \left[ \begin{array}{c} b+\tilde{\ell}_1 \\ \tilde{\ell}_1
\end{array}  \right]  \left[ \begin{array}{c} t+\tilde{\ell}_1 \\ \tilde{\ell}_1
\end{array}  \right] [\tilde{\ell}_1]! [b+\tilde{\ell}_2+ t]! \qquad  \]  
\[ P(\sigma_t^2) =  \left[ \begin{array}{c} b+\tilde{\ell}_2 \\ \tilde{\ell}_2
\end{array}  \right]  \left[ \begin{array}{c} t+\tilde{\ell_2} \\ \tilde{\ell}_2
\end{array}  \right] [\tilde{\ell}_2]! [b+\tilde{\ell}_1+ t]! \qquad  \]  
\[ P(\sigma_t^3) =  \left[ \begin{array}{c} b+\tilde{\ell}_1 \\ \tilde{\ell}_1
\end{array}  \right]  \left[ \begin{array}{c} t+\tilde{\ell}_1 \\ \tilde{\ell}_1
\end{array}  \right] [\tilde{\ell}_1]! [b+\tilde{\ell}_1+t]! \qquad  \]  

Moreover,  if $k_1 + k_2 \leq n$, then the maximum possible number of through strands of $\overline{T_{\sigma}}$ is $k_1 + k_2$, and there are three possible choices of skein elements in the expansions of the center Jones-Wenzl projector which achieve this: $\sigma = \overline{\sigma}_t^1, \overline{\sigma}_t^2$, or $\overline{\sigma}_t^3$  according to whether $k_1 < k_2$, $k_1 > k_2$, or $k_1 = k_2$. 
\begin{figure}[H]
\def \svgwidth{.8\columnwidth}
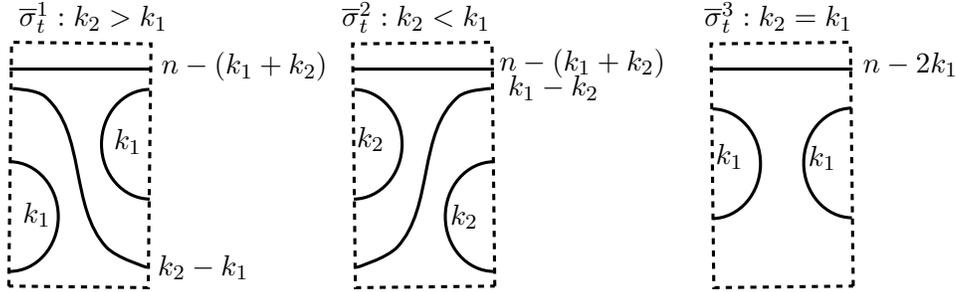
\caption{Choices of skein elements maximizing through strands when $k_1+k_2\leq n$ }
\end{figure}
The coefficients of the choices of skein elements are: $ P(\overline{\sigma}_t^1) = P(\overline{\sigma}_t^2) =  \frac{[n-k_1]![n-k_2]!}{[n-(k_1+k_2)]!}$, and $ P(\overline{\sigma}_t^3) = \frac{([n-k_1]!)^2}{[n-2k_1]!}$. 
\end{lem} 
\begin{proof} The first statement follows directly from examining the possibilities of connecting the endpoints on the dashed box of the center Jones-Wenzl projector. We enumerate the choices of a skein element in its expansion that would not result in a cap or a cup composed with a projector.  The proof of the second statement in the case $k_1+k_2\leq n$ is similar. The formulas for the coefficients are obtained by definition.

\end{proof}

\begin{lem} \label{l.induct} Suppose we have a skein element in $TL_n$ of the following form as shown in Figure \ref{f.idex}, where a skein element $\sigma_t$ is previously chosen in the expansion of the Jones-Wenzl projector in the top dashed box. Further suppose that $\sigma_t$ is one of the choices $\sigma_t^1, \sigma_t^2$, and $\sigma_t^3$ in Lemma \ref{l.fmjw}. 
\begin{figure}[H]
\def \svgwidth{.4\columnwidth} 
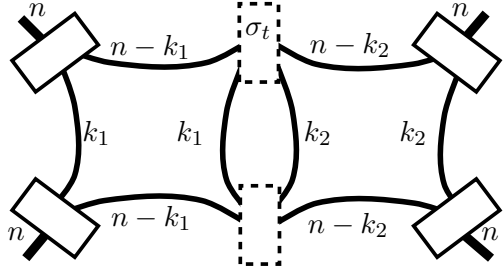
\caption{\label{f.idex} The middle pair of projectors are shown with dashed boxes.}
\end{figure}
\end{lem}
Choosing $\sigma_t$ results in a skein element which may have $c>0$ circles attached to the bottom projector. After removing $c$ circles from the bottom projector via \eqref{e.rcircle}, there are three choices, $\sigma = \sigma_b^1, \sigma_b^2$, and $\sigma_b^3$, up to mirror images via a reflection across the vertical axis, of a skein element in the expansion for the bottom Jones-Wenzl projector in $TL_{n-c}$ (from the empty dashed box in Figure \ref{f.idex}) that does not result in a cap or a cup composed with the four framing Jones-Wenzl projectors.
\begin{figure}[H]
\def \svgwidth{.6\textwidth}
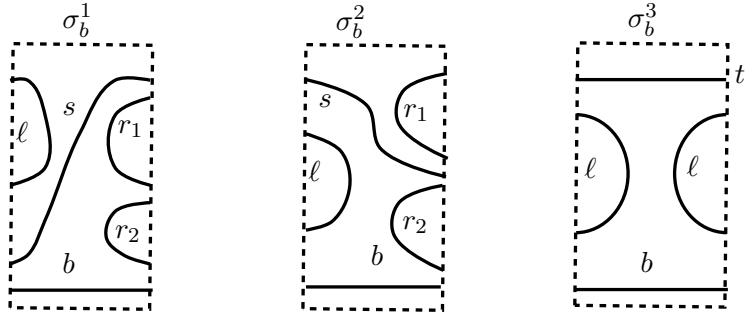
\caption{\label{f.case} All the possibilities for a skein element in the expansion of the bottom projector after choosing a skein element in the expansion for the top projector. Note we must have $r_1 + r_2 = \ell$ for both $\sigma_b^1$ and $\sigma_b^2$.}
\end{figure} 
The coefficient of each choice of skein element $\sigma_b^i$ for $1\leq i \leq 3$ and their mirror images $\breve{\sigma}_b^i$ are: 
\begin{align*} 
P(\sigma_b^1) = P(\breve{\sigma}_b^1) & =  \left[\begin{array}{c} r_1+r_2 \\ r_1 \end{array} \right]  \left[\begin{array}{c} b+\ell \\ \ell \end{array} \right]  [\ell]![b+s+\ell]!, \\ 
 P(\sigma_b^2) = P(\breve{\sigma}_b^2) & =  \left[\begin{array}{c} s+r_1 \\ r_1 \end{array} \right]^{-1}  \left[\begin{array}{c} r_1+r_2 \\ r_1 \end{array} \right]  \left[\begin{array}{c} b+\ell \\ \ell \end{array} \right]  \left[\begin{array}{c} s+\ell \\ \ell \end{array} \right][\ell]![b+s+\ell]!, \text{ and} \\
P(\sigma_b^3) = P(\breve{\sigma}_b^3)&=\left[\begin{array}{c} t+\ell \\ \ell \end{array} \right] \left[\begin{array}{c} b+\ell \\ \ell \end{array} \right] [\ell]![b+t+\ell]!. 
\end{align*} 
\begin{proof}
Given a choice of skein element $\sigma_t=\sigma_t^1$ or $\sigma_t^2$ from Figure \ref{f.glueskein-proj2} for the top center Jones-Wenzl projector, we have the following picture, shown in the middle of Figure \ref{f.rcc}, of the strands coming out of the remaining bottom projector, up to reflection across the vertical axis. 
\begin{figure}[H]
\def \svgwidth{.9\columnwidth}
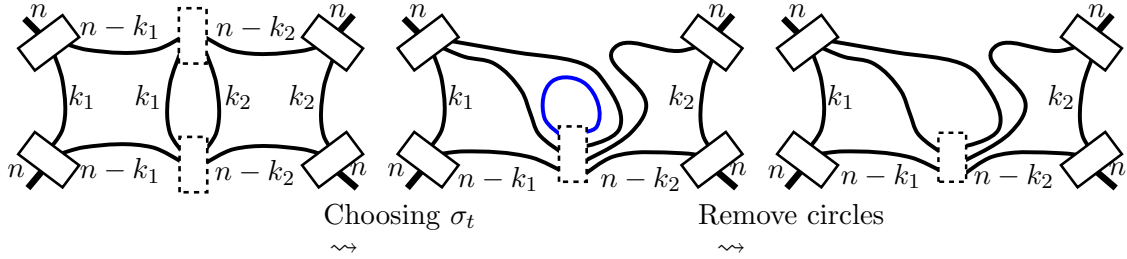
\caption{\label{f.rcc} We remove the circles shown in blue via Equation \eqref{e.rcircle}. } 
\end{figure} 
After removing $c$ circles connected to the bottom projector as the result of choosing $\sigma_t$, we have the picture on the right of Figure \ref{f.rcc}. In order for a choice of a skein element in the expansion of the bottom projector not to result in a cup or a cap composed with one of the four projectors,  we need to connect the strands coming out of an endpoint on the box to another. Examination will show that the only possible cases are $\sigma_b^1$ or $\sigma_b^2$ as shown in Figure \ref{f.case}, up to mirror images via a reflection across the vertical axis. For the choice of $\sigma_t=\sigma_t^3$ from Figure \ref{f.glueskein-proj2}, it is easy to see that the choice of a skein element $\sigma$ for the bottom projector has to be $\sigma_b^3$ as shown in Figure \ref{f.case} using similar arguments. The coefficients are computed for $\sigma_b^1$, $\sigma_b^2$ by applying Lemmas \ref{l.ccup} and \ref{l.Khslide} . 
\end{proof}

\section{The degree of the colored Jones polynomial of pretzel knots} \label{s.degp}
Consider the skein element $T$ comprised of two skein elements $T_1, T_2$ decorated with Jones-Wenzl projectors joined side by side as in the first picture of Figure \ref{f.dt}.  
\begin{figure}[H]
\def \svgwidth{\columnwidth}
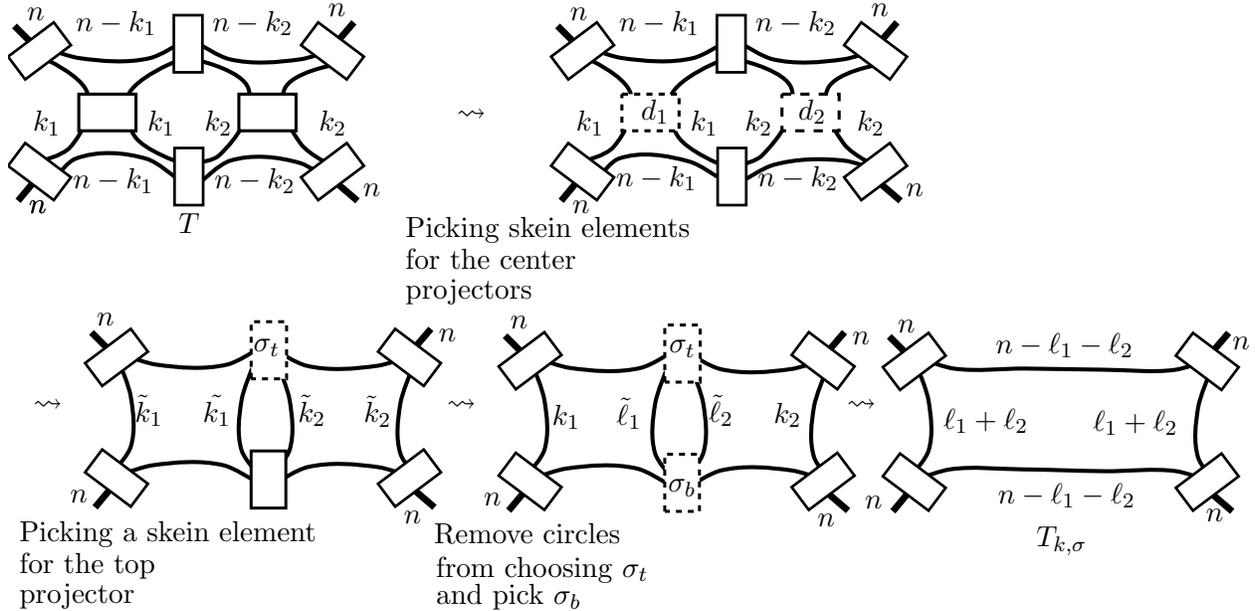
\caption{\label{f.dt} The sequence of choices of expansions for an adjacent set of skein elements}
\end{figure}  We define a sequence $\sigma = (d_1, d_2, \sigma^t, \sigma^b)$ of choices of skein elements for certain Jones-Wenzl projectors decorating $T$. First choose skein elements $d_1, d_2$ in the expansion of the center Jones-Wenzl projectors (marked with dashed boxes)  as in Lemma \ref{l.mjw}. Then, using Lemma \ref{l.fmjw}, pick the skein element $\sigma_t$ in the expansion of the top projector shown with a dashed box. After $\sigma_t$ is chosen, remove any circles attached to the bottom projector via \eqref{e.rcircle}, then use Lemma \ref{l.induct} to pick the skein element $\sigma_b$ in the expansion of the bottom projector shown with a dashed box. Let $k = (k_1, k_2) \in \mathbb{Z}^2_{\geq 0}$. Define $F_{k, \sigma}(q)$ to be the product of rational functions in $q$ resulting from this sequence of moves, replacing projectors by skein elements in the expansion, and removing circles by Equation \eqref{e.rcircle}. 
\begin{equation} F_{k, \sigma}(q) := \P(\sigma_t)\P(\sigma_b)R(c) \prod_{i=1, 2} \P(d_i), \end{equation} 
where $c$ is the number of removed circles, and recall $R(c) = (-1)^c \frac{[n+2]}{[n+2-c]}$. 
Define $T_{k, \sigma}$ to be the skein element resulting from this sequence of moves applied to $T$.
\[T= \sum_{\substack{k, \sigma: \ 0\leq k_i \leq n, \\ n, \ n, \ k_i \text{ admissible}}} F_{k, \sigma}(q) T_{k,\sigma}. \] 

Before proving Theorem \ref{t.mainintro}, we use Lemmas \ref{l.mjw}, \ref{l.fmjw}, and \ref{l.induct} to establish a lemma comparing the degrees of $F_{k, \sigma}(q)$ coming from different choices of skein elements $\sigma = (d_1, d_2, \sigma_t, \sigma_b)$ in the expansions of the Jones-Wenzl projectors. Recall $\overline{T_{k, \sigma}}$ is the skein element obtained by replacing all the Jones-Wenzl projectors of $T_{k, \sigma}$ by the identity skein element $id$. 
\begin{lem} \label{l.main}
Suppose the parameters $k_1 \geq 0 , k_2 \geq 0$ satisfy $k_1+k_2 \geq n$, and a choice of skein elements $\sigma = (d_1, d_2, \sigma_t, \sigma_b)$ is such that $\overline{T_{k, \sigma}}$ has $2(\ell_1+\ell_2)$ through strands. Then there exists a set of parameters $\ell = (\ell_1, \ell_2)$ and a choice of skein elements  $\overline{\sigma} = (\overline{d}_1 = id, \overline{d}_2 = id, \overline{\sigma}_t, \overline{\sigma}_b)$ in the expansions of the corresponding Jones-Wenzl projectors, such that $\overline{T_{\ell, \overline{\sigma}}}$ has the same number of through strands as $\overline{T_{k, \sigma}}$, $\ell_i \leq k_i$ for $i=1, 2$, and $\ell_i \leq \ell_j$ if $k_i \leq k_j$. Moreover, letting $l_1 = k_1 - \ell_1 $ and $l_2 = k_2-\ell_2$  and $\tilde{\ell}_1, \tilde{\ell}_2$ be the intermediate through strands coming from choosing $\sigma_t$ and removing circles attached to the bottom projector, see Figure \ref{f.dt}, we have 
 \begin{equation} \label{e.ineq}
\deg F_{k, \sigma}(q) - \deg F_{\ell, \overline{\sigma}}(q) 
\leq \sum_{i=1}^2 l_i(2\ell_i + l_i). 
\end{equation} 

\begin{figure}[H]
\def \svgwidth{.5\columnwidth} 
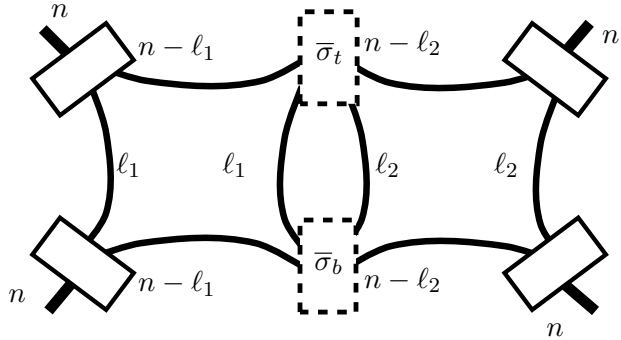
\caption{\label{f.join} A picture of the skein element $T_{\ell, \overline{\sigma}}$. The middle pair of projectors are shown with dashed boxes.}
\end{figure}
\end{lem} 
\begin{proof}
First note that the number of through strands $2(\ell_1+\ell_2)$ of $\overline{T_{k, \sigma}}$ satisfies $2(\ell_1+\ell_2)\leq 2n$. Let $\ell = (\ell_1, \ell_2)$, with $\ell_i \leq k_i$ and such that $\ell_1+\ell_2\leq n$.  By Lemmas \ref{l.fmjw} and \ref{l.induct}, we can always find a sequence of choices of skein elements $\overline{\sigma} = (\overline{d}_1, \overline{d}_2, \overline{\sigma}_t, \overline{\sigma}_b)$ with the number of through strands of $\overline{T_{\ell, \overline{\sigma}}}$ equal to that of  $\overline{T_{k, \sigma}}$ by letting $\overline{d}_1, \overline{d}_2$ be the identity and $\overline{\sigma}_t=\sigma_t^1$ or $\sigma_t^2$ be as in Lemma \ref{l.fmjw}, depending on whether $\ell_2 > \ell_1$ or $\ell_2 < \ell_1$, and $\overline{\sigma}_b$ its mirror image via a reflection across the horizontal axis. 

Suppose for $\sigma = (d_1, d_2, \sigma_t, \sigma_b)$,  the skein element $d_i$ for the middle projector for each twist region $w_i$ is not the identity, then choosing it results in $\tilde{k_1}, \tilde{k_2}$ parameters as in the first figure of the second row of Figure \ref{f.dt}.  Note $\tilde{k_1} + \tilde{k_2}\geq \ell_1 + \ell_2$ and $\deg( P(d_1)P(d_2)) \leq \deg(P(\overline{d}_1)P(\overline{d}_2))$. 

\begin{figure}[H] 
\def \svgwidth{\columnwidth}
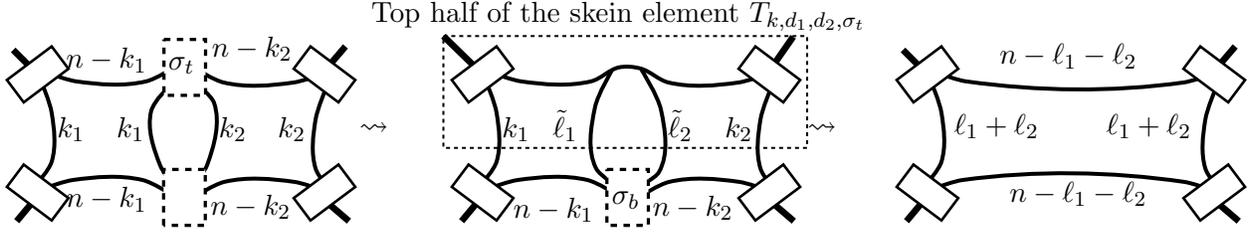
\caption{\label{f.proofdex} The step generating intermediate through strands $\tilde{\ell_1}$, $\tilde{\ell_2}$ after first choosing a skein element $\sigma_t$ on top }
\end{figure} 

In the next step, a skein element $\sigma_t$ is chosen for the top projector,  which will result in intermediate $\tilde{\ell_1}+\tilde{\ell_2} + k_1 + k_2$ through strands for the top half of the skein element $\overline{T_{k, d_1, d_2, \sigma_t}}$, see Figure \ref{f.proofdex}. Note $\tilde{\ell_1}+\tilde{\ell_2}+k_1+k_2 \geq 2(\ell_1+\ell_2)$. If $\tilde{\ell_1}+\tilde{\ell_2}+k_1+k_2 = 2(\ell_1+\ell_2)$, then the choice of the skein element $\sigma_b$ of the bottom projector is necessarily the mirror image of that of $\sigma_t$ via a reflection across the horizontal axis, with top horizontal strands removed. Call this choice of expansion $\hat{\sigma}  = (\hat{d_1} = id, \hat{d_2}=id, \hat{\sigma_t}, \hat{\sigma_b})$. Denote the intermediate through strands of $\hat{\sigma}$ as $\hat{\ell_1}, \hat{\ell_2}$. See Figure \ref{f.proofc} for an example of $\hat{\sigma}_t$. 

We first compare $\sigma$ to $\hat{\sigma}$. Suppose the choice of skein elements $(d_1, d_2, \sigma_t)$ results in intermediate through strands $\tilde{\ell_1}+\tilde{\ell_2}> \hat{\ell_1}+\hat{\ell_2}$ for the top half of the  skein element $\overline{T_{k, d_1, d_2, \sigma_t}}$ (the result of choosing $d_1, d_2, \sigma_t$ for the Jones-Wenzl projectors of $T$ then replacing all the remaining projectors by the identity). Without loss of generality assume $\tilde{\ell}_2 \geq \tilde{\ell}_1$.  Comparing the skein elements $\sigma_t$ and $\hat{\sigma_t}$,  this means  $\tilde{\ell}_1 + \tilde{\ell}_2 > \hat{\ell}_1 + \hat{\ell}_2$ and $\tilde{t} + \tilde{b} < \hat{t} + \hat{b}$. We must have $\hat{\ell}_2 - \hat{\ell}_1 = \tilde{\ell}_2 - \tilde{\ell}_1$ and $\hat{\ell}_1 = \tilde{\ell}_1-l$ for $l>0$.  This is because with fixed parameters $k_1, k_2$ and using Lemma \ref{l.fmjw}, the choice of skein element for $\sigma_t$ corresponding to a given number of through strands is determined by the number of turnbacks (which is $\tilde{\ell}_1$ if $\tilde{\ell}_2>\tilde{\ell}_1$ and $\tilde{\ell}_2$ otherwise).  Then, since $\tilde{\ell_1}+\tilde{\ell_2} > \hat{\ell_1} + \hat{\ell_2}$,  the choice of the skein element $\sigma_b$ must decrease the number of resulting through strands to $\hat{\ell}_1 + \hat{\ell}_2$ so that $\overline{T_{k, \sigma}}$ would still have $2(\ell_1+\ell_2)$ through strands. The possibilities from Lemmas \ref{l.fmjw} and \ref{l.induct} are shown in Figure \ref{f.proofc}.
\begin{figure}[H]
\def \svgwidth{\columnwidth}
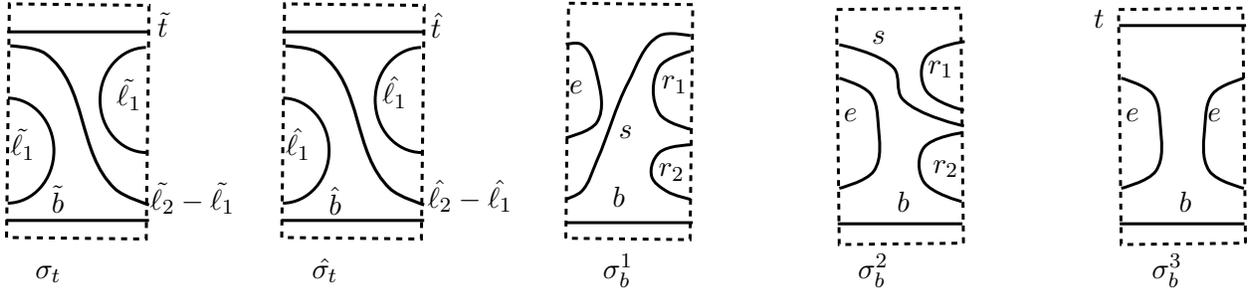
\caption{\label{f.proofc} Left: Comparing $\sigma_t$ and $\hat{\sigma_t}$. Right: Possibilities for $\sigma_b$ corresponding to $\sigma_t$. We changed $\ell$ to $e$ from Lemma \ref{l.induct}.}
\end{figure} 

We show $\deg F_{k, \sigma} \leq \deg F_{k, \hat{\sigma}}$ for all the cases of $\sigma_t$ shown in Figure \ref{f.proofc} and note that the arguments are the same for their mirror images via a reflection across the vertical axis. We compute the difference between $\deg F_{k, \sigma}, \deg F_{k, \hat{\sigma}}$ using the explicit formulas of the quantum factorials. Since the algebraic computations based on the formulas of Lemmas \ref{l.fmjw} and \ref{l.induct} are straightforward, we will omit them and show the results of the computations. To simplify notations, we also let 
$p(x, y, z, t) = \deg P(d)$, where $d$ is the skein element in Lemma \ref{l.khprop}. Note by the preceding discussion,  we know $\tilde{\ell}_1 = \hat{\ell}_1+l,\tilde{b} = \hat{b}-l$, $\tilde{t} = \hat{t}-l$, and $n = \hat{b} + \hat{\ell}_1 + \hat{\ell}_2 +  \hat{t}$. 
\begin{itemize}
\item[\textbf{Case 1: $\tilde{\ell}_2 - \tilde{\ell}_1 = 0$}.] This would force $\sigma_b = \sigma_b^3$ as in Figure \ref{f.proofc}, and  $\hat{\ell}_2 - \hat{\ell}_1 = 0$. We have 
\begin{align*}
&\deg(F_{k, \sigma}) - \deg(F_{k, \hat{\sigma}}) = \deg(\P(\sigma_t) \P(\sigma_b) R(\tilde{b}) \P(d_1)\P(d_2))- \deg(\P(\hat{\sigma}_t) \P(\hat{\sigma}_b) R(\hat{b}) \P(id)\P(id)) \\
&=\underbrace{p(\tilde{b}, \tilde{\ell}_1, 0, \tilde{t})}_{\deg(\P(\sigma_t))} + \underbrace{\tilde{b}}_{\deg R(\tilde{b})}+ \underbrace{p(b, e, 0, 0)}_{\deg(\P(\sigma_b))} - \deg [n-\tilde{b}]!-(\underbrace{p(\hat{b}, \hat{\ell}_1, 0, \hat{t})}_{\deg(\P(\hat{\sigma}_t)\P(\hat{\sigma}_b))} + \underbrace{\hat{b}}_{\deg R(\hat{b})}-\underbrace{p(\hat{b}, \hat{\ell}_1, 0, 0)}_{\deg(\P(\hat{\sigma}_b))}-\deg[n-\hat{b}]!  ) \\
\intertext{We can use the fact that $T_{k, \sigma}$ and $T_{k, \hat{\sigma}}$ have the same number of through strands to get $b = \hat{t}$, $e = \hat{\ell}_1$, and $t = l_1$. Plugging  in and simplifying, we get}
&\deg(F_{k, \sigma}) - \deg(F_{k, \hat{\sigma}})  \leq  -l (1 + l + 2\hat{\ell}_1) \leq 0.
\end{align*} 
since $l, \hat{\ell}_1 \geq 0$ and  $\deg(\P(d_1)\P(d_2)) \leq \deg(\P(id)\P(id))$ . 
\item[\textbf{Case 2: $\tilde{\ell}_2 - \tilde{\ell}_1 \not= 0$}.] Either $\sigma_b = \sigma_b^2$ or $\sigma_b^1$ as in Figure \ref{f.proofc}. 
\begin{itemize}
\item[2(a) $\sigma_b = \sigma_b^2$.] 
\begin{align*} &\deg(F_{k, \sigma}) - \deg(F_{k, \hat{\sigma}}) = \deg(\P(\sigma_t) \P(\sigma_b) R(c) \P(d_1)\P(d_2))- \deg(\P(\hat{\sigma}_t) \P(\hat{\sigma}_b) R(\hat{c}) \P(id)\P(id))  \\ 
&= p(\tilde{b}, \tilde{\ell}_1, \tilde{\ell}_2-\tilde{\ell}_1,  \tilde{t}) + \tilde{b} + \deg \left[ \begin{array}{c} s+r_1 \\ s \end{array} \right]^{-1}+ \deg \left[ \begin{array}{c} r_1+r_2 \\ r_1 \end{array} \right] + p(b, e, 0, s) - \deg [n-\tilde{b}]!\\
&-(p(\hat{b}, \hat{\ell}_1, 0, \hat{t})+ \hat{b}-p(\hat{b}, \hat{\ell}_1, 0, 0)-\deg[n-\hat{b}]!) \\ 
\intertext{Again using the fact that $T_{k, \sigma}$ and $T_{k, \hat{\sigma}}$ have the same number of through strands, we get $b = \hat{t}$,  $e =  \hat{\ell}_2, s = l - (\hat{\ell}_2 - \hat{\ell}_1), r_1 = \hat{\ell}_2 - \hat{\ell}_1$, and $r_2 = \hat{\ell}_1 $. Plug in and simplify:   } 
&\deg(F_{k, \sigma}) - \deg(F_{k, \hat{\sigma}})  \leq  -l (1 + l + 2\hat{\ell}_2) \leq 0. 
\end{align*} 
\item[2(b) $\sigma_b = \sigma_b^1$.] With similar arguments, we get 
\[\deg(F_{k, \sigma}) - \deg(F_{k, \hat{\sigma}})  \leq -l (1 + l + 2\hat{\ell}_2) \leq 0.   \] 
\end{itemize}
\end{itemize}

Thus without loss of generality, to prove the last statement of the lemma, Equation \eqref{e.ineq}, we may assume $d_1 = d_2 = id$, and $\sigma_b$ is the mirror image of $\sigma_t$ via a reflection across the horizontal axis (top horizontal strands removed), with intermediate through strands satisfying $\tilde{\ell}_1 + \tilde{\ell}_2 + k_1 + k_2 = 2(\ell_1 + \ell_2)$. We compare $\sigma_t$ to $\overline{\sigma}_t$.

\begin{figure}[H]
\def \svgwidth{.5\columnwidth}
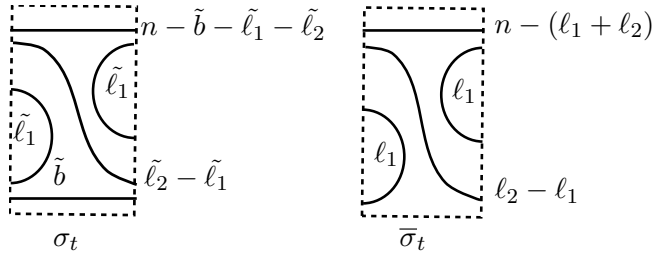
\caption{The case $\ell_2 \geq \ell_1$ and $k_2 \geq k_1$ is shown. }
\end{figure} 
Let $k_1 - \ell_1 = l_1$ and $k_2-\ell_2 = l_2$,  then $\tilde{\ell}_1+l_1 = \ell_1$ and $\tilde{\ell}_2+l_2 = \ell_2$. We have, plugging into the equations for the degrees and assuming $k_1 \leq k_2$ (the other case is similar), 
 \begin{align*}
 \deg(F_{k, \sigma}) - \deg(F_{\ell, \overline{\sigma}}) &\leq k_1-\tilde{\ell}_1 + 2 l_2 \tilde{\ell}_1 + 2 l_1 (l_2 + \tilde{\ell}_2)  = 2l_1 +  2 l_2 \tilde{\ell}_1 + 2 l_1 l_2 + 2l_1 \tilde{\ell}_2 \leq \sum_{i=1}^2 l_i(2\ell_i + l_i). 
 \end{align*} 
 The argument for when $k_1 \geq k_2$ is similar. 
\end{proof} 

We prove a refined version of Theorem \ref{t.mainintro}. We will focus  on the degree of the Kauffman bracket and suppress the monomial from the writhe term in the $n$ colored Jones polynomial. 

\begin{thm} \label{t.ssum} Let $k = (k_0, k_1, \ldots, k_m) \in \mathbb{Z}_{\geq 0}^{m+1}$ with $k_i \leq n$ and $\sigma =  (d_1, \ldots, d_m, \sigma_t^1, \ldots, \sigma_b^{m-1})$ as described in the introduction. For the pretzel link $P=P(w_0, w_1, \ldots, w_m)$ with standard diagram $L$, we have 
\[  \langle L^n \rangle =\sum_{0\leq k_i\leq n, \sigma} G_{k, \sigma} \langle T^n_{k, \sigma} \rangle  =  \sum_{k_0 = \sum_{i=1}^m k_i } \mathcal{G}_{k},\]
where 
\begin{align*} 
 \mathcal{G}_{k} &=\left(\prod_{i=0}^m \frac{[2k_i+1]}{\theta(n, n, 2k_i)} U(w_i, k_i)\right) \left( \prod_{i=1}^{m-1} \left( \frac{[n-\sum_{j=1}^i k_j]![n-k_{i+1}]!}{[n-\sum_{j=1}^{i+1} k_j]![n]!} \right)^2 \right) \langle \mathcal{T}_{k_0} \rangle + l(k)\end{align*} 
Here $l(k)$ denotes lower order terms, and $\langle \mathcal{T}_{k_0} \rangle$ is the Kauffman bracket of the following skein element: 
\begin{figure}[H]
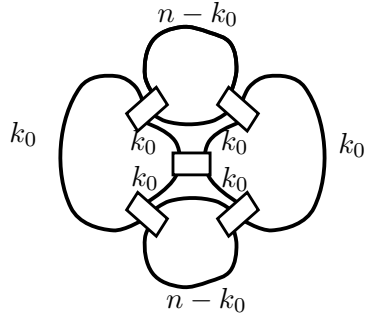

\def \svgwidth{.3\columnwidth}
 
\caption{\label{f.tkop} The skein element $\mathcal{T}_{k_0}$.}
\end{figure}  
Moreover, 
\begin{equation} \label{e.ddiff} 
 \deg \mathcal{G}_{k} - \deg l(k) \geq  2 \min_{0 \leq i \leq m} \{|w_i|-1\} \min_{0\leq i \leq m} \{k_i \}   \end{equation} 
\end{thm} 
\begin{proof}
Recall $T^n_{k, \sigma}$  is the skein element that comes from applying $\sigma$ to $T_k$. This is a skein element of the form as shown below: 
 \begin{figure}[H]
\def \svgwidth{.3\columnwidth}
\begingroup%
  \makeatletter%
  \providecommand\color[2][]{%
    \errmessage{(Inkscape) Color is used for the text in Inkscape, but the package 'color.sty' is not loaded}%
    \renewcommand\color[2][]{}%
  }%
  \providecommand\transparent[1]{%
    \errmessage{(Inkscape) Transparency is used (non-zero) for the text in Inkscape, but the package 'transparent.sty' is not loaded}%
    \renewcommand\transparent[1]{}%
  }%
  \providecommand\rotatebox[2]{#2}%
  \newcommand*\fsize{\dimexpr\f@size pt\relax}%
  \newcommand*\lineheight[1]{\fontsize{\fsize}{#1\fsize}\selectfont}%
  \ifx\svgwidth\undefined%
    \setlength{\unitlength}{482.9785484bp}%
    \ifx\svgscale\undefined%
      \relax%
    \else%
      \setlength{\unitlength}{\unitlength * \real{\svgscale}}%
    \fi%
  \else%
    \setlength{\unitlength}{\svgwidth}%
  \fi%
  \global\let\svgwidth\undefined%
  \global\let\svgscale\undefined%
  \makeatother%
  \begin{picture}(1,0.85109139)%
    \lineheight{1}%
    \setlength\tabcolsep{0pt}%
    \put(0,0){\includegraphics[width=\unitlength,page=1]{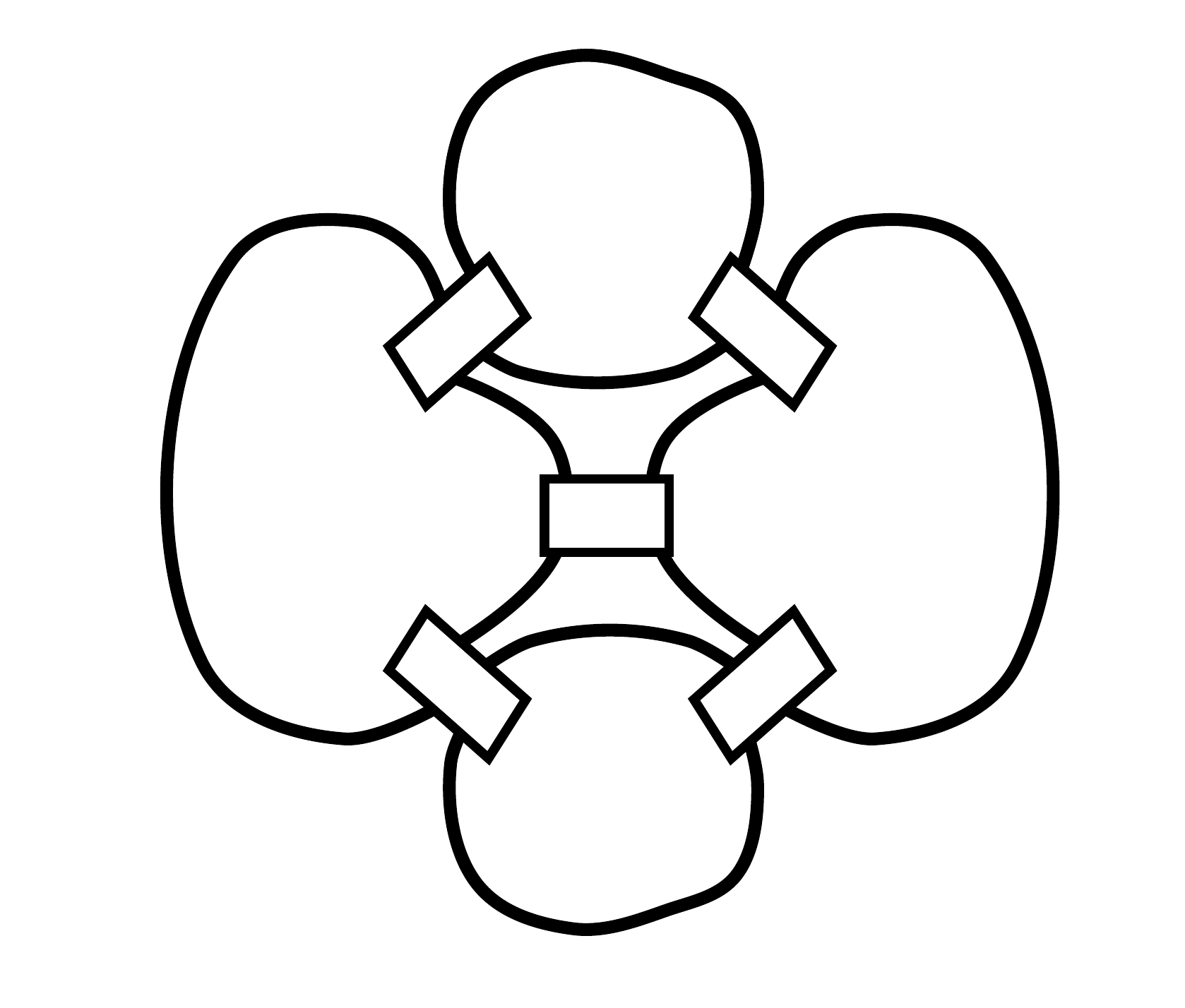}}%
    \put(0.59553863,0.45585509){\makebox(0,0)[lt]{\lineheight{1.25}\smash{\begin{tabular}[t]{l}$k_0$\end{tabular}}}}%
    \put(0.33743744,0.45585509){\makebox(0,0)[lt]{\lineheight{1.25}\smash{\begin{tabular}[t]{l}$k_0$\end{tabular}}}}%
    \put(0.59649247,0.35026046){\makebox(0,0)[lt]{\lineheight{1.25}\smash{\begin{tabular}[t]{l}$k_0$\end{tabular}}}}%
    \put(0.34149703,0.35336619){\makebox(0,0)[lt]{\lineheight{1.25}\smash{\begin{tabular}[t]{l}$k_0$\end{tabular}}}}%
    \put(0.92858438,0.44653795){\makebox(0,0)[lt]{\lineheight{1.25}\smash{\begin{tabular}[t]{l}$c$\end{tabular}}}}%
    \put(-0.00313404,0.47759522){\makebox(0,0)[lt]{\lineheight{1.25}\smash{\begin{tabular}[t]{l}$c$\end{tabular}}}}%
    \put(0.41613918,0.8192253){\makebox(0,0)[lt]{\lineheight{1.25}\smash{\begin{tabular}[t]{l}$n-c$\end{tabular}}}}%
    \put(0.44054132,0.00729927){\makebox(0,0)[lt]{\lineheight{1.25}\smash{\begin{tabular}[t]{l}$n-c$\end{tabular}}}}%
  \end{picture}%
\endgroup%
 
\caption{\label{f.gtko} $T^n_{k, \sigma}$}
\end{figure}  
Note that if $k_0 >  c$ then $T^n_{k, \sigma} = 0$ as proved in \cite{Lee17}. Thus, we can assume $k_0 \leq c$. Let
\[ d(\kappa, w) =\deg \left( \frac{[2\kappa+1]}{\theta(n, n, 2\kappa)} U(w, \kappa) \right).\]
Note that $d(\kappa, w)$ monotonically increases in $\kappa$ if $w < 0$. Therefore we can also assume $k_0 = c$.  We organize the state sum \eqref{e.ssum} by the number of strands $c$ of $T^n_{k, \sigma}$.  A pair $(k, \sigma)$ is \emph{tight} if $c = k_0 = \sum_{i=1}^m k_i$. 

For every pair of parameters and state $(\kappa, \tau)$ with $2c$ through strands for $T^n_{\kappa, \tau}$, there is a tight state $(k(\kappa, \tau), \sigma(\kappa, \tau))$ with the same number of through strands $2c$ and such that $k_i \leq \kappa_i$ and $k_i \leq k_j$ if $\kappa_i \leq \kappa_j$ for all $1\leq i \leq m$ by Lemma \ref{l.main}.  Note $\sigma(k, \tau)$ is determined by $k(\kappa, \tau) = (k_0, k_1, \ldots, k_m)$ as follows: Using Lemma \ref{l.fmjw}, choose $d_i  = id$ for all $1\leq i \leq m$ and choose $\sigma_t^1, \sigma_b^1$ to be the state that will result in $2(k_1 + k_2)$ through strands for $\overline{T_{k, d_1, \ldots, d_m, \sigma_t^1, \sigma_b^1}}$, where $T_{k, d_1, \ldots, d_m, \sigma_t^1, \sigma_b^1}$ is the skein element that comes from applying $d_1, d_2, \sigma_t^1, \sigma_b^1$ to the portion of $T_k$ that is the union of $T_1, T_2$.  Then choose  $\sigma_t^2, \sigma_b^2$ to be the state that will result in $2(k_1+k_2+k_3)$ through strands for $\overline{T_{k, d_1, \ldots, d_m, \sigma_t^1, \sigma_b^1, \sigma_t^2, \sigma_b^2}}$, and so on for the rest of $\sigma_t^i, \sigma_b^i$, which are states that will result in $2(k_1+k_2+\cdots k_{i+1})$ through strands for  $\overline{T_{k, d_1, \ldots, d_m, \sigma_t^1, \sigma_b^1, \sigma_t^2, \sigma_b^2, \ldots, \sigma_t^i, \sigma_b^i}}$,where $T_{k, d_1, \ldots, d_m, \sigma_t^1, \sigma_b^1, \ldots, \sigma_t^i, \sigma_b^i}$ is the skein element that comes from applying $d_1, d_2, \ldots, d_i, \sigma_t^1, \sigma_b^1, \ldots, \sigma_t^i, \sigma_b^i$ to the portion of $T_k$ that is the union of $T_1, T_2, \ldots, T_{i+1}$. It is possible to have more than one $(k(\kappa, \tau), \sigma(\kappa, \tau))$ for a pair $(\kappa, \tau)$, since there could be more than one $k(\kappa, \tau)$ that satisfies the above conditions. 

To prove that every term corresponding to the pair $(\kappa, \tau)$ that is not tight is a lower order term in the sum, we compare the degree of its coefficient function $G_{\kappa, \tau}$ to that of the tight state corresponding to the pair $(k(\kappa, \tau), \sigma(\kappa, \tau))$. Recall the function $G_{\kappa, \tau}$ is the product of coefficients multiplying $\langle T^n_{\kappa, \tau}\rangle$ as in \eqref{e.ssum}. 

Write $k_i = k_i(\kappa, \tau)$ and let $\kappa_i = k_i + l_i$. We get 
\begin{equation} 
d(k_i+l_i, w_i) - d(k_i, w_i) =l_i - l_i(1+2k_i + l_i) w_i.  
\end{equation} 
For an $m$-tangle pretzel link and the skein element $T$, choosing skein elements $\sigma_t^1$ and $\sigma_b^1$ between the pair of skein elements $T_{1} $ and $T_{2}$ creates a new skein element, say $T_{12}$. Similarly, we denote by $T_{123}$ the skein element created by choosing $\sigma_t^2$ and $\sigma_b^2$ between $T_{12}$ and $T_3$, and so on, until $T_{123\cdots (m-1)}$.    Let $F_{1}=\deg F_{ (\kappa_1, \kappa_2),  (\tau_t^1, \tau_b^1)} = \deg(\P(\tau_t^1)\P(\tau_b^1)R(c_1))$, $\overline{F}_{1}=\deg F_{(k'_1, k'_2),  (\overline{\sigma}_t^1, \overline{\sigma}_b^1)} = \deg(\P((\overline{\sigma}_t^1)\P(\overline{\sigma}_b^1))$, where $2(k'_1 +k'_2) = \tilde{\ell}_1 + \tilde{\ell}_2 + \kappa_1 + \kappa_2$ from applying Lemma \ref{l.main} to these skein elements successively. Similarly define $F_{12} = \deg F_{ (\kappa_{12} = k'_1 + k'_2, \kappa_3),  (\tau_t^2, \tau_b^2)} = \deg( \P(\tau_t^2)\P(\tau_b^2)R(c_2))$ and $\overline{F}_{12} = \deg F_{ (k'_{12} = k''_1 + k''_2, k'_3),  (\overline{\sigma}_t^2, \overline{\sigma}_b^2)} = \deg( \P(\overline{\sigma}_t^2)\P(\overline{\sigma}_b^2))$, 
\[ F_{12\cdots(i)}=\deg F_{(\kappa_{1\cdots(i)}=k^{(i-1)}_1 + \cdots + k^{(i-1)}_{i}, \kappa_{i+1}), (\tau_t^{i}, \tau_b^{i})}, \]   
\[ \overline{F}_{12\cdots(i)} = \deg F_{(k'_{1\cdots(i)} = k^{(i)}_1 + \cdots + k^{(i)}_{i}, k'_{i+1}), (\overline{\sigma}_t^{i}, \overline{\sigma}_b^{i})}.\] Writing it all out with $l'_{1\cdots(i)} = \kappa_{1\cdots(i)} - k'_{1\cdots(i)}, l'_i = \kappa_{i}-k'_i$ and applying Lemma \ref{l.main}, we get
\begin{align*}
\deg G_{\kappa, \tau}  - \deg  G_{k(\kappa, \sigma), \sigma(\kappa, \tau)}  &= \left(\sum_{i=0}^m d(k_i+l_i, w_i) -  d(k_i, w_i)\right) + \left(\sum_{i=1}^{m-1} F_{12\cdots(i)} -\overline{F}_{12\cdots(i)}\right),  \text{where} \\
\sum_{i=1}^{m-1} F_{12\cdots(i)} -\overline{F}_{12\cdots(i)} 
&\leq  l'_1(2k'_1 + l'_1)  + l'_2(2k'_2 + l'_2) \\ 
&+ l'_{12}(2k'_{12} + l'_{12}) +   l'_3(2k'_3 + l'_3) \\ 
&+ \cdots \\
&+ l'_{123\cdots(m-1)}(2k'_{123\cdots(m-1)} + l'_{123\cdots(m-1)}) + l'_m(2k'_m + l'_m). \\
\intertext{Since $\sum_{i=1}^{m-1} l'_{12\cdots(i)} = \sum_{i=1}^m l_i$ and $k'_{1\cdots(i)}, k'_i \leq k_i$, we can regroup the sum above and get} 
\deg G_{\kappa, \tau}  - \deg  G_{k(\kappa, \sigma), \sigma(\kappa, \tau)}
&  \leq \left(\sum_{i=1}^m l_i - l_i(1+2k_i + l_i) w_i\right) + \left(\sum_{i=1}^m  l_i(2k_i + l_i)\right)\\
&\leq  -2 \min_{0 \leq i \leq m} \{|w_i|-1\} \min_{0\leq i \leq m} \{k_i \}.   \\ 
\end{align*} 
Therefore, every state corresponding to parameters $(\kappa, \tau)$ that are not tight can grouped into the lower order term $l(k)$ of some state with tight parameters $(k, \sigma)$ in $\mathcal{G}_k$. 
\end{proof}

This gives the following corollary, Theorem \ref{t.genstatesum}, which we restate here for the convenience. 
\gssr*
\begin{proof} 
From Theorem \ref{t.ssum}, we simply compute the degree of the leading term of  $\mathcal{G}_k$ corresponding to the choice of skein elements $k, \sigma$ with tight parameters $ k_0 = \sum_{i=1}^m k_i$. Note $\deg\langle \mathcal{T}_{k_0}\rangle$ is $n$, half the number of circles in $\overline{\mathcal{T}_{k_0}}$, since it is an adequate skein element by \cite{Arm13}.  Thus the degree is 
\begin{align*}
& \underbrace{\sum_{i=0}^m w_i(n-k_i + \frac{n^2}{2}-k_i^2) + \sum_{i=0}^m (k_i-n)}_{\text{ fusion and untwisting}}+ \sum_{i=1}^{m-1} \deg{ \left( \frac{[n-\sum_{j=1}^i k_j]![n-k_{i+1}]!}{[n-\sum_{j=1}^{i+1} k_j]![n]!} \right)^2 } + \underbrace{n}_{\text{number of circles in $\mathcal{T}_{k_0}$}}\\
  &= \frac{2n+n^2}{2}\sum_{i=0}^m w_i - \sum_{i=0}^m w_i k_i - \sum_{i=0}^m w_i k_i^2 + \sum_{i=0}^m k_i -nm - \left( \sum_{i=1}^m k_i\right)^2 + \sum_{i=1}^m k_i^2+n.
 \end{align*} 
 This proves the degree formula in the theorem after regrouping. The sign of the leading coefficient is obtained by multiplying the sign of each of the functions in the product. 
\end{proof}

\subsection{3-tangle pretzel knots $P(w_0, w_1, w_2)$.}  \label{ss.3pretzel} 
We restate Theorem \ref{t.3pretzel} here. 
\pretzel*
 
 \begin{proof}
Let $L$ be the standard diagram of the pretzel knot $P(w_0, w_1, w_2)$. We apply Theorem \ref{t.ssum} to write $\langle L^n \rangle$ as a sum $\sum_{k} \mathcal{G}_{k}$. We compute $\deg\langle L^n\rangle$ by first characterizing the terms $\mathcal{G}_{k}$ that maximize the degree $\delta(n, k)$ (see Theorem \ref{t.genstatesum}), and then determining the degree of the leading term that remains after possible cancellations of power series with the same degree but opposite-sign coefficients. Adding the writhe term finishes the proof.  For $x  = (x_0, x_1, \ldots, x_m) \in \mathbb{R}^{m+1}$, we compute the real maximum of the function $\delta(n, x)$, where $n = x_0 = \sum_{i=1}^m x_i$. Denote the critical points of the real maximum by $x^* = (x_0^*, x_1^*, \ldots, x^*_m)$ Since the leading term of $\deg \mathcal{G}_{k}$ has coefficient $(-1)^{w_0(n-k_0)+n+k_0+\sum_{i=1}^m (n-k_i) (w_i-1)}$ and $w_0, w_2$ are odd while $w_1$ is even, we have that $(-1)^{n-k_1}$ determines the sign of the leading term of $\deg \mathcal{G}_{k}$ with degree $\delta(n, k)$. Therefore, if two parameters $k_1$ and $k_1'$ differ by 1, then $(-1)^{n-k_1}  = -(-1)^{n-k_1'}$, and cancellation of lattice maxima of $\delta(n, k)$ occurs when $x^*_1$ in $x^* = (n, x^*_1, x^*_2)$ is a half integer. Let $k = (k_0 = n, k_1 = x^*_1 + 1/2, k_2= n-k_1)$ with respective choice of the skein element $\sigma$ and $k' =  (k_0 = n, k'_1 = x^*_1 - 1/2 = k_1-1, k_2' = n-k_1')$ with respective choice of skein element $\sigma'$, then $\deg \mathcal{G}_{k} = \deg \mathcal{G}_{k'}$ and their leading terms have opposite signs. This results in cancellation. 
 
 In our case with the 3-pretzel, we get 
 \[ x_1^* = \frac{-2n-w_1+w_2+2n w_2}{2(-2+w_1 + w_2)}.  \]  
 For $x_1^* + \frac{1}{2}$ to be an integer, we must have
 \[ n = -1 + \frac{-2+w_1+w_2}{\gcd(w_1-1, w_2-1)}j\] for a positive integer $j\geq 0$.   
With the quadratic integer programming method applied to the degree $\delta(n, k) = \deg \mathcal{G}_k$, when $n \not= -1+\frac{-2+w_1+w_2}{g} j$, there is no cancellation, we get the degree of $J_{P, n}$ as in the theorem.  See \cite{GLV17}. 
 
Let $g=\gcd(w_1-1, w_2-1)$ and fix $n = -1+\frac{-2+w_1+w_2}{g} j$ for $j\geq 1$. Let $v = (w_2-1)\frac{j}{g}$, and $u$ be a non-negative integer. Consider the pair of states  $(\mathbf{k}_u, \sigma_u)$ and $(\mathbf{k}'_u, \sigma'_u)$ with parameters $\mathbf{k}_u = (n,  v +u,  n-v-u)$ and $\mathbf{k}'_u = ( n,   v-u-1,   n-v + u+1)$, and the respective choices of skein elements $\sigma_u$, $\sigma'_u$ in the expansions of the Jones-Wenzl projectors as specified in Theorem \ref{t.ssum}. These are the leading terms of $\mathcal{G}_{\mathbf{k}_u}$ and $\mathcal{G}_{\mathbf{k}'_u}$, respectively.  We have  
  \begin{align} \label{e.k} 
  &G_{\mathbf{k}_u, \sigma_u}(q) \langle T^n_{\mathbf{k}_u, \sigma_u} \rangle^{-1} ([n]!)^2 \frac{\theta(n, n, 2n)}{[2n+1]}\\ 
   &= q^{w_0(n-\frac{2n}{2}+\frac{n^2}{2}- \frac{(2n)^2}{4})}  \frac{[2(u+v]+1]}{\theta(n, n, 2(u+v))} q^{w_1(n-\frac{(2(v+u))}{2}+\frac{n^2}{2}- \frac{(2(v+u))^2}{4})}   \frac{[2n-2(v+u)+1]}{\theta(n, n, 2n-2(v+u))} \notag  \\
  &\cdot  q^{w_2(n-\frac{2n-2(v+u)}{2}+\frac{n^2}{2} - \frac{(2n-2(v+u))^2}{4})}  
 \left( \frac{[n-(v+u))]![n-(n-(v+u))]!}{[0]!} \right)^2 (-1)^{n-v-u}, \notag \\ 
\intertext{where explicitly} 
& \theta(n, n, 2(v+u)) = \frac{[1+n+v+u]![n-v-u]![v+u]![v+u]!}{[2(v+u)]![n]![n]!}, \text{ and } \notag  \\
& \theta(n, n, 2n-2(v+u)) = \frac{[1+2n-(v+u)]![v+u]![n-(v+u)]![n-(v+u)]!}{[2n-2(v+u)]![n]![n]!} . \notag 
  \end{align} 
Similarly, we get 
  \begin{align}  \label{e.kp} 
  &G_{\mathbf{k}'_u, \sigma'_u}(q) \langle T^n_{\mathbf{k}'_u, \sigma'_u}\rangle^{-1} ([n]!)^2 \frac{\theta(n, n, 2n)}{[2n+1]} \\ 
  &=q^{w_0(n-\frac{2n}{2}+\frac{n^2}{2}-\frac{(2n)^2}{4})} \frac{[2(v-u)-1]}{\theta(n, n, 2(v-u)-2)} q^{w_1(n-\frac{(2(v-u)-2)}{2}+\frac{n^2}{2}- \frac{(2(v-u)-2)^2}{4})}  \notag \\ 
& \cdot \frac{[2n-2(v-u)+2]}{\theta(n, n, 2n-2(v-u)+2)} q^{w_2(n-\frac{(2n-2(v-u)+2)}{2}+ \frac{n^2}{2}- \frac{(2n-2(v-u)+2)^2}{4})} \notag \\ 
&\cdot \left(\frac{[n-((v-u)-1)]![n-(n-(v-u)+1)]!}{[0]!}\right)^2  (-1)^{n-v+u+1} \notag \\ 
\intertext{where}
& \theta(n, n, 2(v-u)-2) = \frac{[n+(v-u)]![1+n-(v-u)]![-1+(v-u)]![-1+(v-u)]!}{[-2 + 2(v-u)]![n]![n]!} \notag \\
& \theta(n, n, 2n-2(v-u)+2) = \notag \\
&\frac{[2+2n-(v-u)]![-1+(v-u)]![1+n-(v-u)]![1+n-(v-u)]!}{[2+2n-2(v-u)]![n]![n]!}.\notag 
  \end{align}  
  A direct computation shows  $\deg{G_{\mathbf{k}_u, \sigma_u}} =  \deg G_{\mathbf{k}'_u, \sigma'_u}$. 
 Factoring out terms that are common to both, we get
 \begin{align*} \label{e.kp} 
&\frac{G_{\mathbf{k}_u, \sigma_u} -G_{\mathbf{k}'_u, \sigma'_u} }{\text{common factors} } \\ 
&= q^{\frac{j}{g}| w_1  -w_2|}\frac{[2v+1][2v][2v-1][2(\frac{j}{g}(w_1+w_2-2)-v)-1]}{[(-3+w_1+w_2)\frac{j}{g} + \frac{g}{j}v+1 ][v]}-  \\ 
&q^{-\frac{j}{g}| w_1  -w_2|} \frac{[2v-1][2(\frac{j}{g}(w_1+w_2-2)-v)+1][2(\frac{j}{g}(w_1+w_2-2)-v)][2(\frac{j}{g}(w_1+w_2-2)-v)-1]}{[\frac{j}{g}(w_1+w_2-2)-v][(-3+w_1+w_2)\frac{j}{g} + w_1 + w_2 - \frac{g}{j}v-1]}. 
\end{align*} 
It is straightforward to see that the degree drop is $2\min \{w_1-1, w_2-1\} \frac{j}{g}$ from $\deg G_{\mathbf{k}_u, \sigma_u}$. Let $L(u) = \deg \left(G_{\mathbf{k}_u, \sigma_u}- G_{\mathbf{k}'_u,  \sigma'_u} \right)$. 
Now we show that all the other terms have degrees that are more than $2\min \{w_1-1, w_2-1\} \frac{j}{g}$ away from $L(0)$. Therefore, $L(0) = \deg G_{\mathbf{k}_0, \sigma_0} - 2\min \{w_1-1, w_2-1\} \frac{j}{g} $ is the degree of $\langle L^n \rangle$. Note first that for a term  $\mathcal{G}_{k}$ of the state sum as in Theorem \ref{t.ssum}, if $k_0 < \sum_{i=1}^2 k_i$, then it has degree strictly smaller by at least $2\min \{w_1-1, w_2-1\} \min \{k'_1, k'_2\}$ compared to $\deg \mathcal{G}_{k'}$, where $k'_0 =\sum_{i=1}^2 k'_i$ by \eqref{e.ddiff} from Theorem \ref{t.ssum}.  Thus it  suffices to compare the degrees of states with tight parameters $k_0 = \sum_{i=1}^2 k_i$. Quadratic integer programming shows that the lattice maxima lie on the diagonal $n = k_0= \sum_{i=1}^2 k_i$ close to the real maximum, and the leading terms corresponding to  $(\mathbf{k}_u, \sigma_u)$ and $(\mathbf{k}'_u, \sigma'_u)$ form a canceling pair as we have seen. In fact, the lattice maxima are the pair of terms corresponding to $(\mathbf{k}_0, \sigma_0)$ and $(\mathbf{k}'_0, \sigma'_0)$, and the next highest degree terms are the terms corresponding to the pair $(\mathbf{k}_u, \sigma_u)$ and $(\mathbf{k}'_u, \sigma'_u)$ for $u> 0$. 
If $j\leq u$, then the degree of such terms are bounded away from the proposed degree by more than $2\min \{w_1-1, w_2-1 \}\frac{j}{g}$ by simply plugging $u$ into the formula for the degree and comparing the resulting quadratics when $u=0$ and $u\not= 0$. We have 
\[ \deg L(0) - 2(w_2-1)\frac{j}{g} > \deg L(u)  \text{ when } (-2+w_1+w_2)u(2+u) > 8\frac{j}{g}(-1+w_2). \]
The inequality is easily satisfied provided $j\gg 0$ and $j\leq u$. 
Finally, all other terms $\mathcal{G}_k$ of the state sum have degrees strictly less than $L(0)-2\min \{w_1-1, w_2-1 \}\frac{j}{g}$ by directly computing the degrees of their leading terms.

 \end{proof}

\bibliographystyle{amsalpha}
\bibliography{references}

\end{document}